\documentclass[11=pt]{amsart}
\usepackage[top = 3cm, left = 3cm, right =3cm, bottom = 4cm]{geometry}
\usepackage{amsthm,amsmath,amssymb,amscd,graphics,enumerate, stmaryrd,xspace,verbatim, epic, eepic,url}
\usepackage[usenames,dvipsnames]{color}
\usepackage{mathtools}
\usepackage[colorlinks=true]{hyperref}
\usepackage[skip=5pt, indent = 10pt]{parskip}

\usepackage[active]{srcltx}
\usepackage{mathrsfs}
\usepackage{tikz-cd}
\usepackage{tikz}
\usetikzlibrary{shapes,shadows,arrows,patterns,trees,automata,positioning,fadings,backgrounds,intersections,calc,decorations.markings}
\tikzstyle{vertex} = [outer color=black,draw, color=black,line width=0.2mm, inner color=black, circle,inner sep=0.2mm,minimum size=0.2mm] 
\tikzstyle{edge} = [line width = 0.3mm] 

\makeatletter
\DeclareRobustCommand{\cev}[1]{%
  {\mathpalette\do@cev{#1}}%
}
\newcommand{\do@cev}[2]{%
  \vbox{\offinterlineskip
    \sbox\z@{$\m@th#1 x$}%
    \ialign{##\cr
      \hidewidth\reflectbox{$\m@th#1\vec{}\mkern4mu$}\hidewidth\cr
      \noalign{\kern-\ht\z@}
      $\m@th#1#2$\cr
    }%
  }%
}
\makeatother

\def\:{\colon}
\def\.{,\dots,}

\def\ZZ{\mathbf Z}

\def\NN{\mathbb N}

\def\Mgn{\overline{\mathcal{M}}_{g,n}}
\def\DMgn{\overline{\mathcal{M}}_{g,A}}
\def\RMgn{\widetilde{\mathcal{M}}_{g,A}}

\def\gp{\textrm{gp}}

\def\gp{\textup{gp}}

\def\DR{\mathsf{DR}^k_{g,A}}
\def\logDR{\mathsf{logDR}^k_{g,A}}

\def\multiset#1#2{\ensuremath{\left(\kern-.3em\left(\genfrac{}{}{0pt}{}{#1}{#2}\right)\kern-.3em\right)}}

{
   \newtheorem{theorem}[subsection]{Theorem}
      \newtheorem*{theorem*}{Theorem}

   \newtheorem{lemma}[subsection]{Lemma}

   \newtheorem{corollary}[subsection]{Corollary}
   
   \newtheorem*{conjecture*}{Conjecture}

}
{\theoremstyle{definition}
          \newtheorem*{exercise*}{Exercise}
   
   \newtheorem{example}[subsubsection]{Example}
   \newtheorem*{example*}{Example}
   \newtheorem{definition}[subsection]{Definition}
   
   \newtheorem*{definition*}{Definition}
   
   \newtheorem{remark}[subsection]{Remark}

}
\title{Smooth Compactifications of the Abel-Jacobi Section}

\author{Sam Molcho \vspace{-3em}}

\begin{document}
\bibliographystyle{amsalpha}
\maketitle
\begin{abstract}
For $\theta$ a small generic universal stability condition of degree $0$ and $A$ a vector of integers adding up to $k(2g-2)$, the spaces $\DMgn^\theta$ constructed in \cite{AP,HMPPS} are observed to lie inside the space $\textbf{Div}$ of \cite{MarcusWise}, and their pullback under $\textbf{Rub} \to \textbf{Div}$ of loc. cit to be smooth. This provides smooth and modular blowups $\RMgn^\theta$ of $\Mgn$ on which the logarithmic double ramification cycle can be calculated by several methods.   
\end{abstract}

\section{Introduction}
The strata of multiscale differentials are the loci 
$$
\{(C,x_1,\cdots,x_n): \omega^k(\sum_{i=1}^n a_ix_i) \cong \mathcal{O}_C\}
$$
in $\mathcal{M}_{g,n}$ for a partition $A = (a_1,\cdots,a_n)$ of integers summing up to $k(2g-2)$. Extensions of these loci to the compactification $\overline{\mathcal{M}}_{g,n}$ have been the subject of a vast literature with different techniques and objectives. In its most algebraic incarnation, such an extension asks for a cycle 
$$
\mathsf{DR}^k_{g,A} \in \mathsf{CH}^g(\overline{\mathcal{M}}_{g,n})
$$
of the expected dimension supported on the ``double ramification locus"
$$
\mathsf{DRL}_{g,A}^k : = \{(C,x_1,\cdots,x_n): \omega^k(\sum a_ix_i) \cong \mathcal{O}_C\} \subset \overline{\mathcal{M}}_{g,n}
$$
of multiscale differentials for the partition $A$. The cycle $\DR$ is called the double ramification cycle, as when $k=0$ it parametrizes functions ramified over two points of $\mathbb{P}^1$, namely zero and infinity. Even the definition of the cycles $\DR$ is subtle; the first rigorous definition was given in \cite{GV} for $k=0$ via the relative Gromov-Witten theory of $\mathbb{P}^1$, and in \cite{Holmes,MarcusWise} in general via Abel-Jacobi theory. Even subtler however is computing the class of $\DR$ in $\mathsf{CH}^g(\Mgn)$; what is meant by computing here is finding an expression of $\DR$ in terms of generators of the tautological ring $\mathsf{R}^*(\Mgn)$. A remarkable such expression, known by now as Pixton's formula, was discovered by Pixton and proven in \cite{JPPZ}.   

Perhaps surprisingly, the developments of \cite{JPPZ} are not the final word to the subject. For instance, if one adopts the Gromov-Witten theory perspective, it is natural to ask for a calculation of the virtual fundamental class for (rubber) relative stable maps to $\mathbb{P}^2$ instead of $\mathbb{P}^1$, or in the multiscale language, for the corresponding classes and calculations of the ``double double" ramification loci 
$$
\{(C,x_1,\cdots,x_n): \omega^k(\sum a_ix_i) \cong \mathcal{O}_C \cong \omega^k(\sum b_ix_i)\}
$$
for two partitions $A,B$. For these problems, the methods of \cite{JPPZ} have not been successfully adapted. To approach them, it has been understood (\cite{HPS},\cite{RanProduct},\cite{HS},\cite{MRan},\cite{Herrproduct}) that one should study these problems in the context of logarithmic geometry. In this context, it is more natural to study instead the \emph{logarithmic} double ramification cycle 
$$
\logDR
$$
This is a certain refinement of $\DR$, but does not live on $\Mgn$ -- or, better, it does not lie in $\mathsf{CH}(\Mgn)$, but rather in the logarithmic Chow ring $\mathsf{logCH}(\Mgn)$ (\cite{BarrottChow},\cite{MPS}). 

We will not define the $\logDR$ here (or the $\DR$ for that matter), but it is possible to explain the relevant aspects of the relationship between $\logDR$ and $\DR$ on general grounds. Let $(X,D)$ be a smooth Deligne-Mumford stack with a normal crossings divisor $D$; the case of primary interest is of course $X = \Mgn, D = \partial \Mgn = \Mgn - \mathcal{M}_{g,n}$. The divisor $D$ then stratifies $X$ into the strata consisting of connected components of intersections $D_1 \cap \cdots D_k$ of various irreducible components of $D$\footnote{With the convention that an irreducible component can repeat if it self intersects.}. A simple blowup is the blowup of $X$ along a smooth stratum closure. Such a blowup $p:X' \to X$ produces a new pair $(X',D' = p^{-1}(D))$. A blowup obtained by iterating this procedure a finite number of times is called an iterated blowup. A \emph{logarithmic blowup} of $(X,D)$ is any blowup $p:X' \to X$ which can be dominated by an iterated blowup of $(X,D)$. Logarithmic blowups form an inverse system, with a map $X'' \to X'$ in the system if the blowup $X'' \to X$ factors through $X' \to X$. In this case we say that $X''$ is finer than $X'$, or a refinement of it. The partial order determined by refinement yields a system of groups $\mathsf{CH}^{\textup{op}}(X')$ indexed by Gysin pullback. Then 
$$
\mathsf{logCH}(X,D) := \varinjlim \mathsf{CH}^{\textup{op}}(X')
$$
where $X' \to X$ ranges through logarithmic blowups of $(X,D)$ \footnote{Alternatively, we can avoid the use of operational Chow rings by restricting attention to $X'$ which are smooth. This gives the same ring as each $X'$ can be dominated by a smooth one.}. The ordinary Chow ring $\mathsf{CH}(X)$ is contained in $\mathsf{logCH}(X)$ as a subring, and there is a retraction(which is not a ring homomorphism) $\mathsf{logCH}(X) \to \mathsf{CH}(X)$ by pushforward. Thus, to say that $\logDR$ is a non-trivial refinement of $\DR$ in $\mathsf{logCH}(\Mgn)$ is to say that $\logDR \notin \mathsf{CH}(\Mgn) \subset \mathsf{logCH}(\Mgn)$ but its pushforward equals $\DR$.   

The ring $\mathsf{logCH}(X,D)$ is, apart from trivial cases, not finitely generated. However, any given element of it is determined by a finite amount of data: for each $x \in \mathsf{logCH}(X,D)$, there exists some log blowup $X' \to X$, and an element $x' \in \mathsf{CH}(X')$ so that $x = x'$ under the natural inclusion $\mathsf{CH}(X') \subset \mathsf{logCH}(X,D)$. Such a pair $(X',x')$ is called a \emph{representative} of $x$ on $X'$. It is however often the case that several such representatives $(X',x')$ exist, with none being preferable: for the sake of concreteness, one could have $(X_i',x_i'),i=1,2$ representing $x$; while by definition there is a representative $(X'',x'')$ dominating both, meaning $p_i:X'' \to X_i'$ is a blowup and $x''=p_i^*(x_i)$, there may be no direct map $X_1 \to X_2$ or vice versa.   

This is the case for $\logDR$. Representatives of it can be found on any blowup $p: \Mgn' \to \Mgn$ which is sufficiently fine, in some sense which we do not make precise here, but which intuitively means that the closure of $\mathsf{DRL}_{g,A}^k$ meets the boundary of $\Mgn'$ sufficiently transversely. The blowups $\Mgn'$ are however neither unique nor canonical, and there is no coarsest or finest blowup supporting a representative. So, in a sense, the ambiguity of choosing a representative is built into the $\logDR$ problem.   

While the ambiguity of representative of a class in $\mathsf{logCH}$ causes few conceptual difficulties, it can cause substantial ones on more practical matters. For instance, if one is interested in writing a formula for the class, several hurdles have to be overcome: for once, one must decide which generating set for the various $\mathsf{CH}^{\textup{op}}(X')$ to use. Fortunately, a good candidate generating set does exist, consisting of the Chow ring $\mathsf{CH}(X)$ and the algebra of boundary strata of the various $X'$, which is captured by combinatorial data: the algebra of piecewise polynomial functions on the tropicalization of $X$ \cite{MPS,MRan,Brion,Payne,FultonSturmfels}. Even so, while this choice of generating set determines the form of the answer, to write down an explicit formula, one generally needs to have precise control over the additional generators adjoined. In practice, this means choosing a representative $(X',x')$ with some sort of special presentation. 

Early approaches to the $\logDR$ focused on properties of the blowup $\Mgn' \to \Mgn$ supporting a representative. The idea here is that, since no best possible choice for $\Mgn'$ exists, one might as well choose one that is fine enough that avoids as many pathologies as possible: choose an $\Mgn'$ that is smooth, whose strata don't self-intersect, and so on. These approaches sufficed to prove soft properties of the $\logDR$, which depend on the form of the class -- it was proven for instance that it is tautological \cite{MRan,HS}. But choosing least pathological models $\Mgn'$ relies on abstract use of resolution of singularities, which makes the problem of  finding an explicit formula essentially impossible. 

In \cite{HMPPS} Pixton's formula was extended to $\logDR$. The strategy adopted there was in the opposite direction: the compactifications $\Mgn'$ constructed were as closely tied with the geometry of the Abel-Jacobi section as possible. The reason to do so was to connect the $\logDR$ with the DR cycle on the universal Picard stack, which had been calculated in \cite{BHPSS} by a (rather elaborate) extension of the methods of \cite{JPPZ}. The end result was, for each ``universal stability condition" $\theta$ \cite{KP}, which through works of \cite{OdaSeshadri,Cap,KP,Pan,Melo} produces a compactified Jacobian $\mathrm{Pic}^{\theta}$, a blowup $\DMgn^\theta \to \Mgn$ which resolves the indeterminacies of the Abel-Jacobi section 
$$
\Mgn \dashrightarrow \mathrm{Pic}^{\theta}
$$
The study of such resolutions was initiated in \cite{AP} via tropical methods, at least in the presence of some mild assumptions on the stability condition; but studying the problem logarithmically allows one to go further, by endowing $\DMgn^\theta$ with an explicit functor of points. In other words, the non-pathological compactifications $\Mgn'$ were traded for \emph{modular} ones. As the functor of points of $\DMgn^\theta$ can be understood completely explicitly when working logarithmically, this was sufficient to compute $\logDR$ on each $\DMgn^\theta$. 

On the other hand, the spaces $\DMgn^\theta$ are typically singular, and the calculation in \cite{HMPPS} expresses $\logDR$ as an operational class. Nevertheless, there are significant advantages to working with a nonsingular space. For instance, in \cite{MRan} the $\logDR$ is approached via strict transforms and Segre classes, and requires as inputs \cite[Theorem 6.7]{FultonInt}, \cite{Aluffi}, which do not work for singular toroidal spaces such as $\DMgn^\theta$. Furthermore, ongoing work of Abreu-Pagani and myself aims to calculate the $\logDR$ by Grothendieck-Riemann-Roch techniques, which require to work with the class in the homological $\mathsf{CH}_*$ instead of the operational theory. For this approach, the singularities of $\DMgn^\theta$ cause difficulties. 

The goal of this paper is to address these difficulties. For each universal stability condition $\theta$, we construct a refinement 
$$
\RMgn^\theta \to \DMgn^\theta
$$
and show 
\begin{theorem}
The stack $\RMgn^\theta$ is smooth. 
\end{theorem}
Furthermore, the refinement $\RMgn^\theta$ is modular:
\begin{theorem}
Let $S$ be a logarithmic scheme, and $S \to \Mgn$ be a logarithmic map, corresponding to a family of curves $C \to S$. Lifts of $S \to \RMgn^\theta$ correspond to pairs $(C' \to C,\alpha)$ consisting of 
\begin{itemize}
    \item A destabilization $C' \to C$,
    \item An equidimensional piecewise linear function $\alpha$ on $C'$ which twists $\omega^k(\sum a_ix_i)$ to a $\theta$-stable line bundle on $C'$. 
\end{itemize}
\end{theorem}
The notion of equidimensional piecewise linear function and twisting is explained in section \ref{sec: tropical moduli}. Stability here is a minimality condition, also explained in \ref{sec: tropical moduli}, which ensures, among other things, finiteness of automorphisms groups. In particular, the strata of $\RMgn^\theta$ are entirely explicit, and correspond to certain combinatorial/linear algebraic data, which we call \emph{$\theta$-stable equidimensional flows}. These are defined in \ref{def: thetastableeqflow}. 

In other words, the space $\RMgn^\theta$ is a ``dream compactification" of the double ramification problem: non-singular and modular. In particular, $\RMgn^\theta$ carries a universal family $\widetilde{C}_{g,A}^\theta \to \RMgn^\theta$, and a universal line bundle $\mathcal{L}$ on $\widetilde{C}_{g,A}^\theta$ -- i.e. an Abel-Jacobi section 
\[
\RMgn^\theta \to \mathrm{Pic}^\theta
\]
The line bundle $\mathcal{L}$ is simply the pullback of the universal line bundle on $\mathrm{Pic}^\theta$. The universal family $\widetilde{C}_{g,A}^\theta$ is not the pullback, but rather a blowup of the pullback of the universal family of $\mathrm{Pic}^\theta$, which is also better behaved: Recall that a scheme is called quasi-smooth if every divisor is $\mathbb{Q}$-Cartier. We have 
\begin{theorem}
The universal curve $\widetilde{C}_{g,A}^\theta$ is quasi-smooth. 
\end{theorem}
From the perspective of the semistable reduction theorems, this result is surprising. The semistable reduction theorem ensures that given a family of curves $C \to S$, we can find blowups $S' \to S$ and $C' \to C \times_S S'$ which are smooth; however, the blowup required on the base $S$ depends on the family $C \to S$, which makes constructing the semistable family $C' \to S'$ explicitly very difficult in practice. 

The connection with the double ramification cycle is as follows: when $\theta$ is sufficiently close to the $0$ stability condition, $\RMgn^\theta$ supports a representative of $\logDR$, and the methods of \cite{HMPPS} also apply: 

\begin{corollary}
The universal $\mathsf{DR}$ formula for $\mathcal{L}$ computes the $\logDR$ on $\RMgn^\theta$. 
\end{corollary}

The smoothness of the spaces $\RMgn^\theta$ however allows also the use of alternative methods of calculation. This was in fact the main driver in writing the paper. Our motivations, in fact, ranked in order of confidence, can be listed as 

\begin{itemize}
\item Find a desingularization of $\DMgn^\theta$ in which one can calculate $\log \mathsf{DR}_{g,A}^k$ via traditional algebro-geometric techniques which avoid Gromov-Witten theory and localization.  
\item Highlight the following phenomenon: $\RMgn^\theta$ is constructed by combining two far away ideas. Stability conditions from the universal Jacobian provide a compact space $\DMgn^\theta$. Techniques from stable maps then provide a desingularization relative to $\DMgn^\theta$. We expect this phenomenon to be present in several moduli problems. 
\item Optimize the computer calculations of $\log \mathsf{DR}_{g,A}^k$\footnote{This possibility was suggested to me by Aaron Pixton.}. Currently, the software deals with the singularities of $\Mgn^\theta$ by desingularizing using a general desingularization algorithm for cones. 
\end{itemize}

The construction of the spaces $\RMgn^\theta$ itself is in fact very simple. In the brilliant work of Marcus and Wise \cite{MarcusWise}, a modification  
$$
\mathbf{Div}_{g,A} \to \Mgn
$$
is constructed. The modification is not of finite type, and highly non-separated, but is the universal modification which resolves the Abel-Jacobi section to the universal Picard stack $\mathbf{Pic}_{g,n}$. Along with it, a further modification 
$$
\mathbf{Rub}_{g,A} \to \mathbf{Div}_{g,A}
$$
is given, which is, up to orbifold corrections a log blowup as above. We briefly review these constructions in section \ref{sec: algebraicmoduli}. The motivation of \cite{MarcusWise} for these constructions is, in the case $k=0$, to compare the double ramification locus 
$$
\mathsf{DRL}_{g,A}^k = \mathbf{Div}_{g,A} \times_{\mathbf{Pic}} \Mgn
$$
with the space of relative rubber maps to $\mathbb{P}^1$, which is identified as 
$$
\mathbf{Rub}_{g,A} \times_{\mathbf{Div}_{g,A}} \mathsf{DRL}_{g,A}^k.
$$
Our observation is simply that the spaces $\DMgn^\theta$ constructed in \cite{HMPPS} are open substacks of $\mathbf{Div}_{g,A}$, and that $\mathbf{Rub}_{g,A}$ is smooth. Combining the two properties gives the spaces $\RMgn^\theta$ as 
$$
\mathbf{Rub}_{g,A} \times_{\mathbf{Div}_{g,A}} \DMgn^\theta.
$$
Most of the work in the paper is translating what the functor of points of this fiber product is (a subtlety is that the fiber product is taken in the category of logarithmic schemes). The main technical tool that allows us to carry out this translation is the determination of the \emph{tropicalization} of $\mathbf{Div}_{g,A}$ and the relevant auxiliary spaces. The tropicalization of $\mathbf{Div}_{g,A}$ is analogous to the tropicalization of the universal Jacobian, as discussed in \cite{MMUV}, and is perhaps of independent interest to tropical geometers. The analogy can in fact be made precise, by recognizing $\mathbf{Div}_{g,A}$ as the pullback of the universal Picard stack to $\Mgn$ via an Abel-Jacobi section, but we do not explore this direction here. 

The presentation in this paper is terse. While we review all the relevant notions that we will use, we do not attempt to give an adequate exposition of them. The reader is meant to have some familiarity with the construction of $\overline{\mathcal{M}}_{g,n}^\theta$ from \cite{HMPPS}, and an idea of the contents of \cite{MarcusWise}.

\subsection*{Acknowledgments} I'm grateful to David Holmes, Nicola Pagani, Rahul Pandharipande, Aaron Pixton, and Johannes Schmitt for several conversations related to double ramification cycles, the spaces $\DMgn^\theta$, and their desingularizations. I'd also like to thank Paolo Aluffi for pointing out to me a broken reference in a previous version of the paper, and explaining to me some points of his work. Special thanks are due to Dhruv Ranganathan and Alex Abreu. Dhruv listened to various versions of the results presented here, from the early to the late stages; our discussions benefited both content and exposition considerably. Alex Abreu's questions on the construction of $\RMgn^\theta$ were numerous and remarkably poignant. They helped me understand several points about the structure of $\RMgn^\theta$ that I had missed, especially regarding the universal family. I'd like to extend my gratitude to both. 

\noindent I was supported by the grant ERC-2017-AdG-786580-MACI.

\section{Stable Twists}
\subsection{Combinatorial Types}
We begin with a short review of some essential notions from tropical geometry. The notions are well-known, but we include them to avoid excessive referencing and to fix notation. Let $\Gamma$ be a graph. We write $V(\Gamma)$ for the set of vertices of $\Gamma$ and $E(\Gamma)$ for its edges. 

\begin{definition}
A \emph{divisor} on $\Gamma$ is a formal $\mathbb{Z}$-linear combination of vertices of $\Gamma$. The group of divisors on $\Gamma$ is denoted by $\mathsf{Div}(\Gamma)$.
\end{definition}
Of course, $\mathsf{Div}(\Gamma)$ is nothing but the free abelian group $\mathbb{Z}^{V(\Gamma)}$ on $V(\Gamma)$, but we insist on the notation for clarity. In situations where other types of divisors are present we sometimes superfluously call divisors on a tropical curve $\Gamma$ tropical divisors.

Let $\mathcal{E}(\Gamma)$ be the set of oriented edges of $\Gamma$. In particular, there is a two-to-one cover $\mathcal{E}(\Gamma) \to E(\Gamma)$, and every choice of orientation $\vec{\Gamma}$ on $\Gamma$ gives a section $E(\Gamma) \to \mathcal{E}(\Gamma)$, whose image we denote by $E(\vec{\Gamma})$. 

\begin{definition}
A flow on $\Gamma$ is a function 
$$
s:\mathcal{E}(\Gamma) \to \mathbb{Z}
$$
satisfying $s(\vec{e}) = - s(\cev e)$. We write $\mathsf{Flow}(\Gamma)$ for the group of flows on $\Gamma$.
\end{definition}
Of course, as above, $\mathsf{Flow}(\Gamma)$ can be identified with $\mathbb{Z}^{E(\vec{\Gamma})}$ after choosing an orientation on $\Gamma$. 

Any flow on $s$ on $\Gamma$ determines a divisor $\mathsf{div}(s)$ on $\Gamma$, by setting 
$$
\mathsf{ord}_v(s) = \sum_{v \in \cev{e}} s(\cev{e})
$$ 
and 
$$
\mathsf{div}(s) = \sum_{v \in V(\Gamma)} (\mathsf{ord}_v(s))v
$$
where by $v \in \cev{e}$ we mean that $e$ is an edge that contains $v$, oriented towards $v$. This procedure gives a homomorphism 
$$
\mathsf{div}:\mathsf{Flow}(\Gamma) \to \mathsf{Div}(\Gamma).
$$

Let $\Gamma' \to \Gamma$ be a subdivision. The additional vertices on $\Gamma'$ are called \emph{exceptional}. A refinement of $\Gamma'$ is a further subdivision $\Gamma'' \to \Gamma'$. 

Suppose $\Gamma' \to \Gamma$ is a subdivision, and $s$ is a flow on $\Gamma'$. It is often desirable to find the minimal subdivision of $\Gamma$ on which $s$ can be defined. 

\begin{definition}
We say that $\Gamma'$ is minimal with respect to $s$ if 

$$
\mathsf{ord}_v(s) \neq 0
$$
on all exceptional vertices $v$ of $\Gamma'$. 
\end{definition}

\begin{lemma}
Suppose $\Gamma'$ is a subdivision of $\Gamma$, and $s$ is a flow on $\Gamma'$. There is a unique minimal subdivision $\Gamma_s \to \Gamma$ on which $s$ can be defined. 
\end{lemma}

\begin{proof}
Define $\Gamma_s$ as the subdivision of $\Gamma$ obtained by keeping only the exceptional vertices of $\Gamma'$ on which 
$$
\mathsf{ord}_v(s) \neq 0.
$$
Since the slope of $s$ changes on the vertices $v$, any subdivision that supports $s$ must refine $\Gamma_s$, whereby the uniqueness of $\Gamma_s$ follows. 
\end{proof}

\subsection{Equidimensionality}  
Let $\Gamma$ be a graph, and $s$ a flow on $\Gamma$. Then $s$ defines a partial orientation on $\Gamma$, by orienting the edges so that $s(\vec{e}) >0$. The orientation is partial, as it is not defined on edges with $s(e) = 0$. We call such edges \emph{contracted}. The flow $s$ defines an honest orientation on the graph $\overline{\Gamma}$ obtained from $\Gamma$ by contracting the contracted edges.

\begin{definition}
A flow $s$ is called acyclic if the graph $\overline{\Gamma}$ has no oriented cycles for the orientation induced by $s$. 
\end{definition}

An acyclic flow $s$ defines a partial order on the vertices of $\Gamma$: the order is generated by the relation 
$$
v < w
$$
if $v,w$ are the endpoints of an oriented edge $\vec{e}$ from $v$ to $w$ in the orientation determined by $s$. Endpoints of contracted edges are not comparable to one another. 

\begin{definition}
Let $\Gamma$ be a graph. An ordering on $\Gamma$ is an acyclic flow $s$, together with the the choice of a total order among the vertices of the non-contracted edges of $\Gamma$, extending the partial order determined by $s$. We say the ordering extends or is subordinate to $s$. 
\end{definition}

The flow $s$ lifts to any subdivision $\Gamma' \to \Gamma$, and an ordering extending $s$ on $\Gamma$ lifts to an ordering extending $s$ on $\Gamma'$ uniquely. We can rephrase the notion of ordering in the following convoluted way, which nevertheless will be meaningful in the next section:
 
\begin{definition}
A one-dimensional combinatorial target, or combinatorial line, is a graph $X$ consisting of $n$-ordered vertices $v_1,\cdots,v_n$, with $v_{i+1}$ joined to $v_i$ by a single edge $e_i$, along with two legs, one on $v_1$ and one on $v_n$. 
\end{definition}

We can then consider maps $\Gamma \to X$. Morphisms for us take cells into cells, i.e. vertices or edges into vertices or edges. A morphism $\phi: \Gamma \to X$ also defines a partial order on $V(\Gamma)$, by declaring $v<w$ if $\phi(v) = v_i$, $\phi(w) = v_{j}$, with $i < j$. Vertices that map to the same vertex of $X$ are incomparable with one another, and we do not define the order on vertices that map into edges of $X$. The following class of morphisms is then special:   

\begin{definition}
A map $\Gamma \to X$ is equidimensional if it takes vertices to vertices. 
\end{definition}

Thus, an equidimensional morphism defines a total order on the vertices of non-contracted edges of $\Gamma$. In fact, an equidimensional map $\Gamma' \to X$ for any subdivision $\Gamma'$ to $\Gamma$ defines a total order on the vertices of non-contracted edges of $\Gamma$. 

\begin{definition}
Let $\Gamma$ be a graph, $s$ an acyclic flow on $\Gamma$. An equidimensional flow extending $s$ is the data of a subdivision $\Gamma' \to \Gamma$, a combinatorial line $X$, and a morphism $\Gamma' \to X$ extending the partial order of $s$. This data is \emph{stable} if all vertices of $X$ are images of vertices of $\Gamma$. 
\end{definition}
We note that stability is not an absolute notion, but depends on the original graph $\Gamma$ on which $s$ is defined.  

Every equidimensional flow extending $s$ defines an ordering extending $s$. Conversely, given an ordering $\kappa$ extending $s$, we can define a combinatorial line $X_\kappa$ by taking one vertex $v_i$ for each vertex of a non-contracted edge of $\Gamma$, according to the order determined by $\kappa$. This defines an evident function
$$
\phi: \Gamma \to X_\kappa
$$
extending $s$. This is however not a morphism: edges of $\Gamma$ can map into unions of cells of $X_\kappa$. There is a minimal subdivision $\Gamma_\kappa$ which turns $\Gamma \to X_\kappa$ into a morphism, by adjoining the preimages of the vertices $\phi^{-1}(v_i)$ to the non-contracted edges of $\Gamma_\kappa$. The following lemma is almost immediate. 
\begin{lemma}
\label{lem:ordering=equid}
Orderings extending $s$ are equivalent to stable equidimensional flows extending $s$, under the correspondence 
$$
\kappa \leftrightarrow \phi: \Gamma_\kappa \to X_\kappa.
$$
\end{lemma}

\subsection{Numerical Stability Conditions}
\label{subsection:NumStabCond}
Let $\Gamma$ be a graph. 
\begin{definition}
A subdivision $\Gamma' \to \Gamma$ is called a \emph{quasi}-stable model of $\Gamma$ if every edge in $\Gamma$ is subdivided at most once. 
\end{definition}

In other words, the subdivision $\Gamma' \to \Gamma$ introduces at most one exceptional vertex on each edge of $\Gamma$. 

\begin{definition}
A divisor $D$ on a quasi-stable model $\Gamma'$ of $\Gamma$ is called admissible if its value on exceptional vertices is $1$. 
\end{definition}

Finally, we recall that a stability condition $\theta$ on $\Gamma$ is simply a function 
$$
\theta: V(\Gamma) \to \mathbb{R}.
$$
It is \emph{non-degenerate} or \emph{generic} if, for every subgraph $S \subset \Gamma$, we have 

$$
\theta(S) \pm \frac{E(S,S^c)}{2} \notin \mathbb{Z}
$$
where $E(S,S^c)$ is the valence of the image of $S$ in the contraction $\Gamma/S$. 

A stability condition determines a list of semistable divisors on $\Gamma$: those $D$ for which 

$$
\theta(S) - \frac{E(S,S^c)}{2} \le D(S) \le \theta(S) + \frac{E(S,S^c)}{2}.
$$
The divisor is stable if the inequalities are strict. Thus, a stability condition is non-degenerate if and only if all semistable divisors are stable. 

A stability condition on $\Gamma$ lifts canonically to a quasi-stable model $\Gamma' \to \Gamma$ by declaring its value on exceptional vertices to be $0$.

\begin{definition}
Let $\theta$ be a stability condition on $\Gamma$, and $\Gamma' \to \Gamma$ a quasi-stable model. We call an admissible divisor $D$ on $\Gamma'$ $\theta$-semistable if for every subgraph $S \subset \Gamma'$, we have
$$
\theta(S) - \frac{E(S,S^c)}{2} \le D(S) \le \theta(S) + \frac{E(S,S^c)}{2}.
$$
\end{definition}
We note that if $\theta$ is generic, the inequalities above are strict for every divisor supported on $V(\Gamma) \subset V(\Gamma')$. However, equality can hold for divisors that have support on exceptional vertices.\footnote{and in fact necessarily holds for $S = \left \{v \right \}$, where $v$ is an exceptional vertex.}

We thus arrive at the key combinatorial notions of this paper. Let $A$ denote a fixed tropical divisor on $\Gamma$, and $\theta$ a stability condition. 

\begin{definition}
A $\theta$-flow balancing $A$ (or $\theta$-flow for short) consists of a quasi-stable model $\Gamma' \to \Gamma$, a $\theta$-semistable divisor $D$, and an acyclic flow $s$ with
$$
\mathsf{div}(s) = A - D.
$$
\end{definition}

\begin{definition}
\label{def: thetastableeqflow}
A $\theta$-stable equidimensional flow (balancing $A$) consists of 
\begin{itemize}
\item A quasi-stable model $\Gamma' \to \Gamma$. 
\item A $\theta$-semistable divisor $D$ on $\Gamma'$. 
\item An acyclic flow $s$ on $\Gamma'$ balancing $D$: 
$$
\mathsf{div(s)} = A-D.
$$
\item A stable equidimensional flow $\Gamma'' \to X$ extending $s$. 
\end{itemize}

\subsection{Specialization}
The data discussed above specializes with respect to edge contractions. Namely, if $\overline{\Gamma}$ is obtained from $\Gamma$ by contracting some vertices, 

\begin{enumerate}
\item Divisors on $\Gamma$ specialize to divisors on $\overline{\Gamma}$ by $D \to \overline{D}$
$$
\overline{D}(v) = \sum_{w} D(w)
$$ 
where $w$ ranges through all vertices of $\Gamma$ contracting to $v \in V(\overline{\Gamma})$. 
\item Stability conditions specialize exactly analogously as $\overline{\theta}(v) = \sum \theta(w)$. 
\item Flows specialize by $s \to \overline{s}$, with 
$$
\overline{s}(\vec{e}) = s(\vec{e})
$$
under the natural inclusion $\mathcal{E}(\overline{\Gamma}) \subset \mathcal{E}(\Gamma)$. 
\end{enumerate}
All notions discussed, starting with subdivisions and culminating with $\theta$-stable equidimensional flows, specialize through under these definitions. 

\end{definition}

\section{Abel-Jacobi Theory on Tropical Curves}
\label{sec: tropical}
The notions of the previous section are combinatorial. We extend them to tropical notions by introducing a metric on our graphs. We recall our convention: monoids $M$ are sharp (they have no non-trivial units), finitely generated, integral and saturated. They are dual to the category of cones, which consist of \emph{rational polyhedral cones} together with an integral structure, under the correspondence 

\begin{align*}
M \to M^{\vee}:= (\mathsf{Hom}(M,\mathbb{R}_{\ge 0}),\mathsf{Hom}(M,\mathbb{Z}))\\
\sigma^{\vee} = \mathsf{Hom}(C,\mathbb{R}_{\ge 0}) \cap \mathsf{Hom}(N,\mathbb{Z}) \leftarrow \sigma = (C,N).
\end{align*}

Let $\Gamma$ be a tropical curve metrized by a monoid $M$. In other words, $\Gamma$ is a graph together with a length function 
$$
\ell:E(\Gamma) \to M-0
$$
from its set of edges to the non-zero elements of $M$. We denote the length of the edge $e$ by $\ell_e$. 

Divisors and flows are combinatorial data and do not take into account the metric structure of $\Gamma$. Piecewise linear functions on $\Gamma$ on the other hand are honest tropical notions:

\begin{definition}
A piecewise linear function $\alpha$ is a function 
$$
\alpha: V(\Gamma) \to M^\gp
$$
from the vertices of $\Gamma$ to the associated group of $M$, which satisfies the following condition: for every oriented edge $\vec{e}$ in $\Gamma$ between $v,w \in \Gamma$, there exists an integer $s(\vec{e}) \in \ZZ$ such that 
$$
\alpha(w) - \alpha(v) = s(\vec{e})\ell_e.
$$
We write $\mathsf{PL}(\Gamma)$ for the group of piecewise linear functions on $\Gamma$. 
\end{definition}

Every piecewise linear function $\alpha$ on $\Gamma$ determines a flow $s_\alpha$ by taking its underlying slopes: 
\begin{align*}
\mathsf{PL}(\Gamma) \to \mathsf{Flow}(\Gamma)\\
\alpha \to s_{\alpha}, s_\alpha(\vec{e}) = \textup{ slope of } \alpha \textup{ on } \vec{e}. 
\end{align*}
In particular, we can talk about divisors of piecewise linear functions, orientations, and so on, via the underlying flow. The flows that can arise from a piecewise linear function are constrained by the metric structure on $\Gamma$. 

\begin{definition}
We call a flow that arises as the underlying slopes of a piecewise linear function a \emph{twist} on $\Gamma$. 
\end{definition}

All flows that arise from piecewise linear functions are acyclic. The condition a flow must satisfy to be a twist on the other hand must involve the metric somehow. In short, we start with a flow $s$ on $\Gamma$ and want to lift it to a function $\alpha$. We can start at a vertex $v \in \Gamma$, and assign a value of $\alpha(v) \in M^\gp$ arbitrarily. But then, the rest of the values $\alpha(w)$ are completely determined by the lengths of $\Gamma$ and the slopes of $s$ (provided $\Gamma$ is connected): for any oriented path $P_{v \to w}$ from $v$ to a vertex $w$, we must have 

$$
\alpha(w) = \alpha(v) + \sum_{\vec{e} \in P_{v \to w}} s(\vec{e})\ell_e
$$
and the function $\alpha$ is well defined if and only if this expression is independent of path. This condition is most conveniently phrased in terms of the \emph{intersection pairing} 
\begin{align*}
\left \langle, \right \rangle : \mathsf{Flow}(\Gamma) \times \mathsf{Flow}(\Gamma) \to M^\gp \\
\left \langle s, t \right \rangle = \frac{1}{2}\sum_{\vec{e} \in 
\mathcal{E}(\Gamma)} s(\vec{e})t(\vec{e})\ell_e.
\end{align*}
In terms of the intersection pairing, the lifting problem amounts to the statement that, for every $\gamma \in H_1(\Gamma)$, we have 
$$
\left \langle s, \gamma \right \rangle = 0.
$$
Here, $H_1(\Gamma)$ is considered as embedded in $\mathsf{Flow}(\Gamma) \cong \mathbb{Z}^{E(\vec{\Gamma})}$, by writing a cycle $\gamma$ as an oriented path 
$$
\gamma = \sum \vec{e}
$$
and associating to $\gamma$ the twist defined by 
$$
\gamma(\vec{e}) = 1.
$$
if $\vec{e}$ is in the path, and $0$ otherwise. 

\begin{remark}
This inclusion identifies $H_1(\Gamma)$ with the kernel of 
$$
\mathsf{div}:\mathsf{Flow}(\Gamma) \to \mathsf{Div}(\Gamma).
$$
This coincides with the usual identification of $H_1(\Gamma)$ with the kernel of 
$$
\mathbb{Z}^{E(\vec{\Gamma})} \to \mathbb{Z}^{V(\Gamma)}
$$
coming from the CW-complex structure on $\Gamma$. 
\end{remark}

\subsection{Subdivisions of Tropical Curves}
Let $\Gamma$ be a tropical curve. A subdivision of $\Gamma$ is a tropical curve $\Gamma'$ metrized by $M$, with a map $\Gamma' \to \Gamma$ so that the sum of the lengths of the edges $e'$ that map to the same edge $e \in \Gamma$ add up to the length of $e$: if $\phi:E(\Gamma') \to E(\Gamma)$ is the induced map of edges, we must have 
$$
\sum_{e' \in \phi^{-1}(e)} \ell_{e'} = \ell_e.
$$
The collection of vertices in $\Gamma'$ that are not in $\Gamma$ are called \emph{exceptional vertices}. A \emph{refinenent} of $\Gamma' \to \Gamma$ is a further subdivision $\Gamma'' \to \Gamma'$. 

Suppose $\Gamma' \to \Gamma$ is a subdivision, and $\alpha$ is a piecewise linear function on $\Gamma'$. It is often desirable to find the minimal subdivision of $\Gamma$ on which $\alpha$ can be defined. 

\begin{definition}
We say that $\Gamma'$ is minimal with respect to $\alpha$ if 

$$
\mathsf{div}(\alpha)(v) \neq 0
$$
on all exceptional vertices $v$ of $\Gamma'$. 
\end{definition}

\begin{lemma}
Suppose $\Gamma'$ is a subdivision of $\Gamma$, and $\alpha$ is a piecewise linear function on $\Gamma'$. There is a unique minimal subdivision $\Gamma_\alpha \to \Gamma$ on which $\alpha$ can be defined. 
\end{lemma}

\begin{proof}
Define $\Gamma_\alpha$ as the subdivision of $\Gamma$ obtained by keeping only the exceptional vertices of $\Gamma'$ on which 
$$
\mathsf{div}(\alpha)(v) \neq 0.
$$
Since the slope of $\alpha$ changes on the vertices $v$, any subdivision that supports $\alpha$ must refine $\Gamma_\alpha$, whereby the uniqueness of $\Gamma_\alpha$ follows. 
\end{proof}

\subsection{Equidimensional Piecewise Linear Functions and Twists}Any integral monoid $M \subset M^\gp$ can be regarded as a partial order on $M^\gp$: for $x,y \in M^\gp$, we declare $x \le y$ if $y - x \in M$. Thus, a piecewise linear function $\alpha$ on $\Gamma$ comes with a partial ordering of its values $\alpha(v) \in M^\gp$. This partial order is evidently compatible with the orientation on $\Gamma$ induced by the underlying twist of $\alpha$. 

\begin{definition}
Let $\Gamma$ be a tropical curve. A piecewise linear function is totally ordered if its values $\alpha(v)$, $v \in V(\Gamma)$, are totally ordered. 
\end{definition}

Borrowing ideas from the theory of semistable reduction, we make the following definition.

\begin{definition}
A piecewise linear function $\alpha$ on $\Gamma$ is \emph{equidimensional} if 
\begin{itemize}
\item The values $\alpha(v)$ are totally ordered. 
\item For any edge $e$ with endpoints $v,w$ that satisfy $\alpha(v)<\alpha(w)$, and any vertex $u$ with $\alpha(v) \leq \alpha(u) \leq \alpha(w)$, we necessarily have $\alpha(u) = \alpha(v)$ or $\alpha(u) = \alpha(w)$. 
\end{itemize}
\end{definition} 

While on its face the definition of equidimensionality seems dependent on the values of $\alpha$, the definition is in fact invariant under translation of the values $\alpha(v)$ by any common element $x \in M^\gp$. Thus, the definition descends to twists, and it makes sense to talk about equidimensional twists. 

The definition of equidimensional piecewise linear function can perhaps be clarified by introducing its image, which is a \emph{one dimensional tropical target}, also referred to as a tropical line. 

\begin{definition}
A tropical line is the structure of a one-dimensional polyhedral complex metrized by $M$ on $\mathbb{R}$. 
\end{definition}
Thus, a tropical line is, as a polyhedral complex, simply a combinatorial line $X$, with a length assignment $\ell_{e_i} \in M-0$ for each of its edges. But to say that this polyhedral complex is a polyhedral complex structure on $\mathbb{R}$ over $M$ means that it furthermore comes with a chosen piecewise linear embedding 
$$
\iota_X: X \subset \mathbb{R}.
$$
This data is very similar to the definition of a piecewise linear function on a tropical curve: an element 
$$
\iota_X(v_i) := \gamma_i \in M^\gp
$$
for each vertex $v \in V(X)$, such that 
$$
\gamma_{i+1}- \gamma_{i} = \ell_{e_i}.
$$
In particular, $\gamma_{i+1}>\gamma_i$. We will also consider the trivial polyhedral decomposition as an allowable tropical line, where $\mathbb{R}$ is considered as a single cell. In that case, we will simply write $\mathbb{R}$ for the tropical line. 

Let $\Gamma$ be a tropical curve and $X$ a tropical line metrized by $M$. By definition, a map of polyhedral complexes $\Gamma \to X$ is a piecewise linear map that respects the cell structure of the polyhedral decomposition: each cell (that is, vertex or edge) of $\Gamma$ maps into a cell of $X$ (rather than a union of more than one cells). In particular, a piecewise linear function $\alpha$ on $\Gamma$ can be tautologically thought of as a map 
$$
\Gamma \to \mathbb{R}
$$
to the trivial tropical line. The piecewise linear function $\alpha$ may or may not factor through $X$: 

\[
\begin{tikzcd}
\Gamma \ar[rr,dashed,"\beta"] \ar[rd,"\alpha"] & & X \ar[ld,"\iota_X"] \\ & \mathbb{R} &  
\end{tikzcd}
\]
Since $\iota_X$ is a monomorphism, the arrow $\beta$, if it exists, is unique. 

\begin{definition}
A map of polyhedral complexes $P \to Q$ is called equidimensional if it takes cells \emph{onto} cells.
\end{definition}
As tropical curves and tropical lines are particularly simple examples of polyhedral complexes, the meaning of equidimensionality of a map 
$$
\beta: \Gamma \to X
$$
is very simple to describe: it says that each vertex $v \in \Gamma$ must map to a vertex in $X$. 

\begin{lemma}
Let $\Gamma$ be a tropical curve metrized by $M$, and let $\alpha$ be a piecewise linear function on $\Gamma$. Then $\alpha$ is equidimensional if, and only if, there exist a tropical target $X$ and a factorization of $\alpha$ through an equidimensional map $\beta: \Gamma \to X$. 
\end{lemma}

\begin{proof}
Suppose $e$ is an edge of $\Gamma$ with endpoints $v,w$, and $\alpha(w) > \alpha(v)$. Suppose $u$ is a vertex of $\Gamma$ with $\alpha(v) \le \alpha(u) \le \alpha(w)$. If $\alpha$ factors through $X$, then $\beta(v)$ and $\beta(w)$ must be consecutive vertices $v_{i},v_{i+1}$ of $X$ (as, otherwise, the edge $e$ would map to a union of cells). But to say that $\beta$ is equidimensional is to say that $\beta(u)$ must be a vertex of $X$, and hence one of $v_{i},v_{i+1}$. So either $\alpha(u)=\alpha(v)$ or $\alpha(u)=\alpha(w)$. Conversely, given an equidimensional function $\alpha$, we build $X$ by taking one vertex $v_i$ for each distinct value in \{$\alpha(v): v \in V(\Gamma)\}$, and define 
$$
\iota_X(v_i) = \alpha(v)
$$
to be the corresponding value. 
\end{proof}

\begin{remark}
Note that while vertices of $\Gamma$ must map to vertices of $X$, edges can map either onto edges of $X$ or vertices. Edges that map onto vertices of $X$ are precisely those on which the slope of $\alpha$ is $0$, i.e. the \emph{contracted} edges. 
\end{remark}

\begin{remark}
The definition of equidimensionality may seem convoluted from the vantage of tropical curves metrized by monoids, but it is natural from the dual point of view of cone complexes and semistable reduction. We find the dual point of view far more intuitive, but, the tropical perspective is more ubiquitous in the literature.  Namely, a tropical curve $\Gamma$ metrized by $M$ is equivalent data to a map of cone complexes (with integral structure) 
$$
\Sigma_C \to \Sigma_S := M^{\vee}
$$
Plainly, one builds $\Sigma_C$ out of $\Gamma$ as a fibration over $\Sigma_S$. The fiber over $x \in \Sigma_S$ is obtained by attaching a vertex $v_x$ for each $v \in V(\Gamma)$, and an edge of length $l_e(x) \in \mathbb{R}_{\ge 0}$ for each $e \in E(\Gamma)$. Then, a piecewise linear function $\alpha$ on $\Gamma$ corresponds to a piecewise linear map 
\[
\begin{tikzcd}
\Sigma_C \ar[rr] \ar[dr]& & \Sigma_S \times \mathbb{R} \ar[dl,"\textup{pr}_1"]\\
& \Sigma_S&
\end{tikzcd}
\]
A tropical line $X$ corresponds to a subdivision $X \to \Sigma_S \times \mathbb{R} \to \Sigma_S$ so that the composed map $X \to \Sigma_S$ is equidimensional (maps cells onto cells) and furthermore sends integral structures onto integral structures; and an equidimensional piecewise linear functions is a factorization of $\alpha$ through a tropical line $X$, that sends cells of $\Sigma_C$ onto cells of $\Sigma_X$. The name equidimensional comes from the fact that maps of fans which send cones onto cones are the ones that induce equidimensional maps of toric varieties. 
\end{remark}

Suppose $\alpha$ is a piecewise linear function or twist on $\Gamma$, and $\Gamma' \to \Gamma$ is a subdivision. Then $\alpha$ lifts to a piecewise linear function on $\Gamma'$. Suppose $\alpha$ is equidimensional on $\Gamma'$. We say that $\Gamma'$ is a \emph{minimal} subdivision on which $\alpha$ is equidimensional if the following stability condition holds: \\

$(\star)$ For every exceptional vertex $v$ of $\Gamma'$, there exists a non-exceptional vertex $w$ of $\Gamma$ such that $\alpha(v) = \alpha(w)$. \\

We remark that the minimal $\Gamma'$ on which $\alpha$ is equidimensional is in general finer than the minimal model $\Gamma_\alpha$ on which $\alpha$ is defined. Furthermore, the model is unique if it exists. This is very similar to the combinatorial analogues in \ref{lem:ordering=equid}, but the tropical picture deviates here: a model where $\alpha$ becomes equidimensional may not exist, as the required combinatorial subdivisions may not lift respect the metric structure. Experts will recognize here that one can always find a minimal equidimensional model, but only after altering the base monoid $M$. However, one can say with relative ease: 

\begin{lemma}
Suppose that $\Gamma$ is a tropical curve with piecewise linear function $\alpha$, and $\Gamma'$ is a subdivision on $\Gamma$ on which $\alpha$ is equidimensional. Then a minimal model $\mathsf{Eq}_{\Gamma}(\alpha)$ on which $\alpha$ is equidimensional exists. 
\end{lemma}

\begin{proof}
As in the construction of $\Gamma_\alpha$, one obtains $\mathsf{Eq}_\Gamma(\alpha)$ from $\Gamma'$ by deleting all exceptional vertices in $\Gamma'$ that violate $(\star)$. 
\end{proof}

\begin{definition}
Suppose $\alpha$ is a piecewise linear function on $\Gamma$, and that $\alpha$ lifts to an equidimensional function 
$$
\Gamma' \to X
$$
on some subdivision of $\Gamma$. We call the lift stable if $\Gamma' \to X$ is the minimal such lift, i.e. satisfies $(\star)$. We write 
$$
\alpha: \mathsf{Eq}_{\Gamma}(\alpha) \to X
$$
for the minimal lift. 
\end{definition}

\section{Tropical Moduli}
\label{sec: tropical moduli}
This section closely follows \cite{MarcusWise} in spirit, but in the tropical world. We make the simplifying assumption that all graphs that appear have at least one leg.  

\begin{definition}
We define stacks
$$
\Sigma_{\textbf{Div}}, \Sigma_{\textbf{Ord}}, \Sigma_{\textbf{Rub}}: \textbf{Mon} \to \textbf{Sets}
$$
by 
\begin{align*}
\Sigma_{\textbf{Div}}(M) = \{\Gamma, \alpha \in \mathsf{PL}(\Gamma)\} \\
\Sigma_{\textbf{Ord}}(M) = \{\Gamma, \alpha \in \mathsf{PL}(\Gamma), \textup{an ordering of }\alpha(v) \}\\
\Sigma_{\textbf{Rub}}(M) = \{\Gamma, \alpha \in \mathsf{PL}(\Gamma), \mathsf{Eq}_\Gamma(\alpha) \to X \}
\end{align*}
where 
\begin{itemize}
    \item $\Gamma$ is metrized by $M$. 
    \item $\alpha$ vanishes on the vertex of $\Gamma$ containing the leg. 
\end{itemize}
Isomorphisms are isomorphism of graphs that respects the functions, orderings, and equidimensional lifts. 
\end{definition}

\begin{remark}
The assumption that $\Gamma$ contains a leg is not serious, but we impose it to rigidify the problems above via the condition $\alpha=0$ on the vertex. Otherwise, we have to talk about tropical line torsors, which we'd rather avoid as all applications we have in mind already have a leg.
\end{remark}

Let $\textbf{CombDiv}$ be the category whose objects consist of a graph $\Gamma$ and an acyclic flow $s$. A map $(\Gamma,s) \to (\overline{\Gamma},\overline{s})$ in $\textbf{CombDiv}$ is given by a map $f: \Gamma \to \overline{\Gamma}$, where $\overline{\Gamma}$ is an edge contraction of $\Gamma$, such that $\overline{s} \circ f = s$ (in particular, automorphisms respecting the flow are allowed). Similarly, we define $\textbf{CombOrd}$ to consist of pairs $(\Gamma,s)$ and a total ordering $\kappa$ on $\Gamma$ extending $s$; and $\textbf{CombRub}$ to consist of a pair $(\Gamma,s)$ and a stable equidimensional lift $\Gamma' \to X$. 

For each $(\Gamma,s) \in \textbf{CombDiv}$, define a cone 
$$
\sigma_{(\Gamma,s)} \subset \mathbb{R}_{\ge 0}^{E(\Gamma)} 
$$
consisting of the $\ell$ that satisfy the equations 
$$
\left \langle s,\gamma \right \rangle_{\ell} =0 
$$
for all $\gamma \in H_1(\Gamma)$. The pairing here is the intersection pairing of section \ref{sec: tropical}. By definition, the pairing requires a length on each edge of $\Gamma$. The subscript $\ell$ here means that on the point $\ell = (\ell_{e})_{e \in E(\Gamma)}$ of $\mathbb{R}_{\ge 0}^{E(\Gamma)}$, we give the tautological length $\ell_e$ to the edge $e$. Under a morphism $(\Gamma,s) \to (\overline{\Gamma},\overline{s})$, we get a face inclusion 
$$
\sigma_{(\overline{\Gamma},\overline{s})}<\sigma_{(\Gamma,s)}
$$
and so we may glue
$$
\Sigma_{\textbf{Div}}' : = \varinjlim_{(\Gamma,s)} \sigma_{(\Gamma,s)}.
$$
Similarly, for a triple $(\Gamma,s,\Gamma' \to X)$ in $\textbf{CombRub}$, we take the cone 
$$
\sigma_{(\Gamma,s,\Gamma' \to X)} \subset \mathbb{R}_{\ge 0}^{E(\Gamma')} \times \mathbb{R}_{\ge 0}^{E(X)}
$$
consisting of the $(\ell_{e},\ell_f)_{e \in E(\Gamma),f \in E(X)}$ that satisfy 
$$
s(e)l_e = l_f
$$
and glue to a cone stack 
$$
\Sigma_{\textbf{Rub}}':=\varinjlim_{(\Gamma,s,\Gamma \to X)} \sigma_{(\Gamma,s,\Gamma' \to X)}.
$$
Finally, for $(\Gamma,s,\kappa) \in \mathbf{CombOrd}$, we take the cone 
$$
\sigma_{(\Gamma,s,\kappa)}
$$
defined by the equations 
$$
\left \langle s, \gamma \right \rangle_\ell = 0
$$
for $\gamma \in H_1(\Gamma)$, and the additional condition: choose a minimal vertex $v_0$ for the ordering $\kappa$, and for each $w \in \Gamma$ an oriented path $P_{v_0 \to w}$. Then, we keep the $\ell=(\ell_e)_{e \in E(\Gamma)}$ for which additionally 
$$
\sum_{\vec{e} \in P_{v \to w}} s(\vec{e})\ell_e \le \sum_{\vec{e} \in P_{v \to w'}} s(\vec{e})\ell_e 
$$
whenever $w \le w'$ in the total order determined by $\kappa$\footnote{We note that the equations $\left \langle s, \gamma \right \rangle_\ell = 0$ imply that the condition does not depend on the choice of $v_0$.}. We glue these cones to a cone stack 
$$
\Sigma_{\textbf{Ord}}' = \varinjlim_{(\Gamma,s,\kappa)} \sigma_{(\Gamma,s,\kappa)}.
$$
\begin{theorem}
The cone stack $\Sigma_{\textbf{Div}}'$ represents the functor $\Sigma_{\textbf{Div}}$, the cone stack $\Sigma_{\textbf{Ord}}'$ represents $\Sigma_{\textbf{Ord}}$, and the cone stack $\Sigma_{\textbf{Rub}}'$ represents $\Sigma_{\textbf{Rub}}$. 
\end{theorem}

\begin{proof}
Let $M$ be a cone, with $M^\vee = \sigma$. An element $\Sigma_{\textbf{Div}}(M)$ is a tropical curve $\Gamma$ metrized by $M$, together with a piecewise linear function $\alpha$ vanishing on the vertex containing the leg. This data defines a map  
\begin{align*}
\sigma \to \mathbb{R}_{\ge 0}^{E(\Gamma)} \\
x \to (\ell_e(x))
\end{align*}
and an underlying flow $s_\alpha:=s$. But the flow is a twist, and so 
$$
\left \langle s,
\gamma \right \rangle \in M^\gp
$$
is $0$. So the $\ell_e(x)$ map into $\sigma_{(\Gamma,s)}$. Conversely, a map $f: \sigma \to \sigma_{(\Gamma,s)}$ defines a metric on $\Gamma$, by taking $\ell_e \in M$ to be the composition
\[
\begin{tikzcd}
\sigma \ar[r,"f"] & \sigma_{(\Gamma,s)} \ar[r,"\textup{pr}_e"] & \mathbb{R}_{\ge 0}
\end{tikzcd}
\]
of $f$ with the $e$-th projection. The acyclic flow $s$ is a twist on $\Gamma$, since the lengths have been chosen so that the equations $\left \langle s,\gamma \right \rangle$ are satisfied. The twist only lifts to a piecewise linear function on $\Gamma$ up to translation by an element of $M^\gp$, but lifts uniquely if we assume that its value is $0$ on the marking containing the leg. This shows
$$
\Sigma_{\textbf{Div}} \cong \Sigma_{\textbf{Div}}'.
$$
The proofs for $\textbf{Ord},\textbf{Rub}$ are similar. 
\end{proof}

There are evident maps $\Sigma_{\textbf{Rub}} \to \Sigma_{\textbf{Ord}} \to \Sigma_{\textbf{Div}}$ obtained by forgetting the additional structure at each step.  

\begin{lemma}
The map $\Sigma_{\textbf{Ord}} \to \Sigma_{\textbf{Div}}$ is a subdivision. 
The map $\Sigma_{\textbf{Rub}} \to \Sigma_{\textbf{Ord}}$ is a a finite index inclusion. 
\end{lemma}
\begin{proof}
We first look at $\Sigma_{\textbf{Ord}} \to \Sigma_{\textbf{Div}}$. It is clear that the map is a monomorphism, so it suffices to show that the map is bijective on $\mathbb{N}$ points. Given a tropical curve $\Gamma$ with integer lengths, and an $\alpha: \Gamma \to \mathbb{R}$, the values of $\alpha$ are elements of $\mathbb{Z}$, so automatically totally ordered. 

For the map $\Sigma_{\textbf{Rub}} \to \Sigma_{\textbf{Ord}}$, it suffices to show that the map is bijective on $\mathbb{R}_{\ge 0}$ points. We start with an $\alpha$ and an ordering $\kappa$ of its values, and build a tropical line $X$ by taking one vertex $v_i$ for each distinct value $\alpha(v)$, with the ordering of $\kappa$. We define an embedding $\iota_X: X \to \mathbb{R}$ by 
$$
\iota_X(v_i) = \textup{ value of corresponding } \alpha(v).
$$
There is a map of topological spaces $\Gamma \to X$, but does not respect cell structures, as interior points of edges of $\Gamma$ map to vertices of $X$. We refine $\Gamma$ to $\Gamma'$ obtained by subdividing along the preimages of the vertices of $X$. This gives a stable equidimensional PL function on $\alpha$. 
\end{proof}

\begin{remark}
\label{rem:rootnecessary}
The argument in the proof essentially shows how the finite index inclusion in $\Sigma_{\textbf{Rub}} \to \Sigma_{\textbf{Ord}}$ arises. When subdividing $\Gamma$ to $\Gamma'$, the new points may have rational coordinates; for instance, if an edge $e$ from $v$ to $w$ has slope $s(e)$, and $\alpha(u)$ is an intermediate value between $\alpha(v)$ and $\alpha(w)$, the function $\alpha$ hits $\alpha(u)$ at the point 
$$
\frac{\alpha(u)-\alpha(v)}{s(e)}
$$
of $e$, and thus $e$ needs to be subdivided there. This point is in $M_{\mathbb{Q}}^\gp$, but not necessarily in $M^\gp$. 
\end{remark}

\begin{lemma}
The cone stack $\Sigma_{\textbf{Ord}}$ is simplicial and the cone stack $\Sigma_{\textbf{Rub}}$ is smooth.  
\end{lemma}

\begin{proof}
It suffices to check that the cones $\sigma_{(\Gamma,s,\Gamma' \to X)}$ of $\Sigma_{\textbf{Rub}}$ are isomorphic to $\mathbb{N}^k$ for some $k$. But the equations 
$$
s(e)\ell_e = \ell_f
$$
show that the cone is in fact isomorphic to 
$$
\mathbb{N}^{E(X)} \times \mathbb{N}^{E^c(\Gamma)}
$$
where $E^c(\Gamma)$ denotes the set of contracted edges of $\Gamma$. The fact that $\Sigma_{\textbf{Ord}}$ is simplicial follows from the fact that its cones are isomorphic to those of $\Sigma_{\textbf{Rub}}$ after tensoring with $\mathbb{Q}$. 
\end{proof}

In fact, more can be said. From the isomorphism of the real points of a cone in $\Sigma_{\textbf{Ord}}$ with $\mathbb{R}_{\ge 0}^{E(X)} \times \mathbb{R}_{\ge 0}^{E^c(\Gamma)}$ it follows that the rays -- i.e. the one dimensional cones -- in $\Sigma_{\mathbf{Ord}}$ are precisely those that parametrize maps $\alpha$ consisting of 
\begin{itemize}
\item Either a single contracted edge. 
\item Or, maps without contracted components to a target with exactly one edge. 
\end{itemize}
In either case, the map $\alpha$ automatically lifts to an equidimensional one. Thus, the sublattice structure of $\Sigma_{\textbf{Rub}}$ agrees with that of $\Sigma_{\textbf{Ord}}$ on rays. Since $\Sigma_{\textbf{Ord}}$ is simplicial, we obtain  

\begin{corollary}
\label{cor: sublatticestructure}
Let $\sigma_{(\Gamma,s,\kappa)}$ be a cone in $\Sigma_{\textbf{Ord}}$, and $\sigma_{(\Gamma,s,\Gamma' \to X)}$ the corresponding cone in $\Sigma_{\textbf{Rub}}$. The lattice of $\sigma_{(\Gamma,s,\Gamma' \to X)}$ is the lattice freely generated by the primitive vectors along the rays of $\sigma_{(\Gamma,s,\kappa)}$.
\end{corollary}
\begin{proof}
Since $\sigma_{(\Gamma,s,\Gamma' \to X)}$ is smooth, i.e. isomorphic to some $\mathbb{N}^k$, its lattice must be the lattice generated by the primitive vectors along its rays. Since those primitive vectors are the same as the primitive vectors of $\Sigma_{\textbf{Ord}}$, the conclusion follows. 
\end{proof}

 Thus, $\Sigma_{\textbf{Ord}}$ is in a certain sense a coarse moduli space of $\Sigma_{\textbf{Rub}}$: see e.g. \cite[Subsection 3.2]{GMtor}.

\subsection{Carving out Small Subcomplexes}
The complexes $\Sigma_{\textbf{Div}},\Sigma_{\textbf{Rub}},\Sigma_{\textbf{Ord}}$ are very large: they are indexed by additional data over the category of all graphs, which is itself notoriously large. When algebraizing, as we will in the next section, the resulting schemes have infinitely many connected components and are highly non-separated. Here, we want to impose several increasingly stringent conditions that carve out subcomplexes which algebraize to much more pleasant spaces.

First, we can as usual restrict our attention to genus $g$, $n$-marked graphs. Writing fixed number of legs $h_1,\cdots,h_n$ on our graphs; we may then write 

$$
\Sigma_{\textbf{Div}_n}
$$
which decomposes 
$$
\Sigma_{\textbf{Div}} = \coprod_{n \ge 1} \Sigma_{\textbf{Div}_{n}}.
$$
Similar descriptions are available for the order and rubber versions. 

Given an $n$-marked graph $\Gamma$, we can consider $n+1$-marked graphs $\Gamma_{c}$, one for every cell $c$ of $\Gamma$ (vertex or edge), obtained by attaching a leg $h_{n+1}$ on $c$ when it is a vertex, and a bivalent vertex with a single leg $h_{n+1}$ on $c$ when it is an edge. A map $\Gamma_{c} \to \Gamma$ is obtained by deleting the latter vertex. There is a map 
$$
\mathbb{R}_{\ge 0}^{E(\Gamma_c)} \to \mathbb{R}_{\ge 0}^{E(\Gamma)}
$$
which is an isomorphism when $c=v$ is a vertex, and which is the fiber product 
\[
\begin{tikzcd}
\mathbb{R}_{\ge 0}^{E(\Gamma_c)} \ar[r] \ar[d] & \mathbb{R}_{\ge 0}^2 \ar[d,"+"] \\ \mathbb{R}_{\ge 0}^{E(\Gamma)} \ar[r,"\ell_e"] & \mathbb{R}_{\ge 0}
\end{tikzcd}
\]
when $e$ is an edge. Given a piecewise linear function $\alpha$ on $\Gamma$, with underlying flow $s=s_\alpha$, it lifts canonically to a piecewise linear function on $\Gamma_c$, by not changing the slopes if $e$ has been subdivided. The union of the cones 

$$
\sigma_{(\Gamma,s,n+1)} \subset \Sigma_{\mathbf{Div}_{n+1}}
$$
forms a subcomplex with a forgetful map to $\sigma_{(\Gamma,s,n)}$.
\begin{lemma}
The universal family of $\Sigma_{\mathbf{Div}_{n}}$ restricts to 
$$
\varinjlim_{c \in V(\Gamma) \cup E(\Gamma)} \sigma_{(\Gamma_c,s,n+1)}
$$
over $\sigma_{(\Gamma,s,n)}$.
\end{lemma}

\begin{proof}
Let $\mathcal{C} \to \Sigma_{\textbf{Div}_n}$ be the universal family, and $\sigma$ a cone. A map $\sigma \to \mathcal{C}$ is a map $\sigma \to \Sigma_{\textbf{Div}_n}$, together with a section of 
$$
\mathcal{C} \times_{\Sigma_{\mathbf{Div}_n}} \sigma \to \sigma.
$$
In other words, it consists of a tropical curve $\Gamma$ metrized by $M = \sigma^\vee$, a piecewise linear function $\alpha$, and a section of $\Gamma$. The section is a point of $\Gamma$, which is either a vertex or an edge. Sections that land in $e$ are in bijection with subdivisions $\Gamma'$ of $\Gamma$ whose underlying graph is $\Gamma_e$, and are thus parametrized by the choice of lengths of the two pieces of $e$ determined by the section. Since $\alpha$ lifts canonically to $\Gamma'$ by not altering the slopes, the triples $(\Gamma,\alpha,\textup{section through } e)$ are in bijection with maps 
$$
\sigma \to \sigma_{(\Gamma_e,s,n+1)}
$$
as claimed. 
\end{proof}

The analogous result does not hold for $\Sigma_{\textbf{Ord}}$. The reason is that an ordering $\kappa$ does not lift canonically to an ordering to the universal curve. All that can be said is that the universal curve of $\Sigma_{\textbf{Ord}}$ is the pullback of the universal curve of $\Sigma_{\textbf{Div}}$.

The result is also \emph{false} for $\Sigma_{\textbf{Rub}}$, but a curious intermediate statement can be obtained. Points of $\Sigma_{\textbf{Rub}}$ have more structure, contained in the map $\Gamma' \to X$. Suppose $\kappa$ is the induced ordering. The section of the universal curve in particular factors through some cell $\Gamma'_c$ of $\Gamma'$ now, and thus the ordering $\kappa$ lifts canonically to $\Gamma'_c$: points in any cell of $\Gamma'$ are in a unique order relative to $\alpha$. It follows that the universal curve of $\Sigma_{\textbf{Rub}_n}$ factors through $\Sigma_{\textbf{Ord}_{n+1}} \to \Sigma_{\textbf{Div}_{n+1}}$. On the other hand, it does not necessarily factor through $\Sigma_{\textbf{Rub}_{n+1}}$. The reason is that while $\alpha$ and the total ordering lift canonically to $\Gamma'_c$, the equidimensional lift 
$$
\Gamma' \to X
$$
does not. The induced map 
$$
\Gamma'_{c} \to X
$$
is no longer equidimensional. Further subdivision of $X$  and consequently of $\Gamma'_c$ is required, which may require extracting additional roots, as in \ref{rem:rootnecessary}. Nevertheless, the argument suffices to show 

\begin{lemma}
\label{lem:univfamsimplicial}
The universal family of $\Sigma_{\textbf{Rub}}$ is simplicial. 
\end{lemma}

\begin{remark}
To get a smooth universal family, one can work instead with an alternative stack $\Sigma_{\textbf{AF}}$\footnote{The acronym stands for ``Abramovich-Fantechi".} parametrizing orbifold tropical curves. We do not introduce this here as $\Sigma_{\textbf{Rub}}$ is good enough for our purposes. 
\end{remark}

\begin{figure}
\caption{Above, a point of $\Sigma_\textbf{Ord}$, representing a piecewise linear function $\alpha$ with $\alpha(v)<\alpha(u)<\alpha(w)$, and a lift to its universal family. The relation between $\alpha(u)$ and $\alpha(v_{n+1})$ on the universal family is undetermined. Below, an analogous point of $\Sigma_{\textbf{Rub}}$. Here the relation $\alpha(v_{n+1})<\alpha(u)=\alpha(u')$ is forced.}
\begin{tikzpicture}
\centering
\draw [thick,domain=0:360] plot ({(3/2)*cos(\x)},{sin(\x)});
\draw[<-,thick](2,0)--(3,0);
\draw [thick,domain=0:360] plot ({5+(3/2)*cos(\x)},{sin(\x)});

\coordinate(v1) at (0,-1){};
\coordinate(v2) at (-1.5,0){};
\coordinate(w2) at (1.5,0){}; 
\coordinate(w1) at (5,-1){};
\coordinate(v3) at (3.5,0){};
\coordinate(w3) at (6.5,0){}; 
\coordinate(w4) at (5+3/2*0.707,0.707){};

\node[draw,circle,inner sep=1pt, fill] at (v1){};
\node[above] at (v1){$u$};
\node[draw,circle,inner sep=1pt, fill] at (v2){};
\node[left] at (v2){$w$};
\node[draw,circle,inner sep=1pt,fill] at (w2){};
\node[right] at (w2){$v$};
\node[draw,circle,inner sep=1pt, fill] at (v3){};
\node[right] at (v3){};
\node[draw,circle,inner sep=1pt,fill] at (w3){};
\node[left] at (w3){};
\node[draw,circle,inner sep=1pt, fill] at (w1){};
\node[above] at (w1){$u$};
\node[draw, red, circle,inner sep=1pt, fill] at (w4){};
\node[above] at (w4){$v_{n+1}$};

\draw [thick,domain=0:360] plot ({(3/2)*cos(\x)},{sin(\x)-4});
\draw[<-,thick](2,-4)--(3,-4);
\draw [thick,domain=0:360] plot ({5+(3/2)*cos(\x)},{sin(\x)-4});

\coordinate(v1) at (0,-5){};
\coordinate(v2) at (-1.5,-4){};
\coordinate(w2) at (1.5,-4){}; 
\coordinate(w1) at (5,-5){};
\coordinate(v3) at (3.5,-4){};
\coordinate(w3) at (6.5,-4){}; 
\coordinate(w4) at (5+3/2*0.707,0.707-4){};
\coordinate(v5) at (0,-3){};
\coordinate(w5) at (5,-3){};

\node[draw,circle,inner sep=1pt, fill] at (v1){};
\node[below] at (v1){$u$};
\node[draw,circle,inner sep=1pt, fill] at (v2){};
\node[left] at (v2){$w$};
\node[draw,circle,inner sep=1pt,fill] at (w2){};
\node[right] at (w2){$v$};
\node[draw,circle,inner sep=1pt, fill] at (v3){};
\node[right] at (v3){};
\node[draw,circle,inner sep=1pt,fill] at (w3){};
\node[left] at (w3){};
\node[draw,circle,inner sep=1pt, fill] at (w1){};
\node[below] at (w1){$u$};
\node[draw, red, circle,inner sep=1pt, fill] at (w4){};
\node[above] at (w4){$v_{n+1}$};
\node[draw, blue, circle,inner sep=1pt, fill] at (v5){};
\node[above] at (v5){$u'$};
\node[draw, blue, circle,inner sep=1pt, fill] at (w5){};
\node[above] at (w5){$u'$};

\end{tikzpicture}
\end{figure}
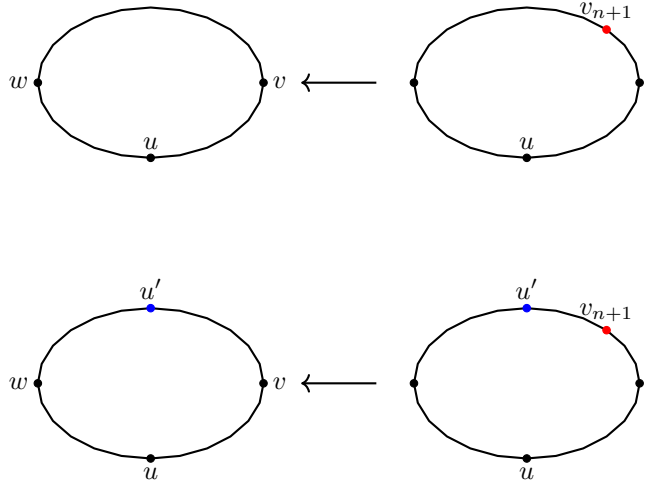

We now continue our carving mission much more aggressively. We first fix the genus of the graph, along with its $n$ markings, to obtain stacks
$$
\Sigma_{?,g,n}.
$$
Next, we fix a \emph{universal stability condition} $\theta$ \cite{KP}. This means a stability condition for the universal family $\mathcal{C}_{g,n}^\textup{trop} \to \Mgn^{\textup{trop}}$ of stable tropical curves: a numerical stability condition in the sense of subsection \ref{subsection:NumStabCond} for each stable graph $\Gamma$ of genus $g$ with $n$ legs, which are compatible with respect to all contractions of edges and automorphisms. We restrict to stable graphs here for simplicity but in fact a similar procedure even works for $\mathfrak{M}_{g,n}^{\textup{trop}}$. We write  $\Sigma_{\mathbf{Div}_{g,n}^\theta}
$ for the subcomplex of $\Sigma_{\mathbf{Div}_{g,n}}$ consisting of the cones $\sigma_{(\Gamma,s)}$ such that the graph $\Gamma$ is quasi-stable, and $D = A - \mathsf{div}(s)$ is $\theta$-stable. In other words, $(\Gamma,s)$ is a $\theta$-flow relative to the stabilization $\Gamma^{\mathsf{st}}$. We write 

$$
\Sigma_{\mathbf{Ord}_{g,A}^\theta}
$$
for the subcomplex consisting of $(\Gamma,s,\kappa)$ with $\Gamma$ quasi-stable and $D = A - \mathsf{div}(s)$ $\theta$-stable, and 
$$
\Sigma_{\mathbf{Rub}_{g,A}^\theta}
$$
for the subcomplex consisting of $(\Gamma,s,\Gamma' \to X)$ with $\Gamma$ quasi-stable, $D = A - \mathsf{div}(s)$ $\theta$-stable, and $\Gamma' \to X$ a stable equidimensional lift. In other words, $(\Gamma,s,\Gamma' \to X)$ is a $\theta$-stable equidimensional flow relative to the stable graph $\Gamma^{\mathsf{st}}$. 

\begin{remark}
We note that the quasi-stable graph $\Gamma$ can be recovered from the destablization $\Gamma'$ and the map $s: \Gamma' \to X$: its vertices are the unstable vertices in $\Gamma'$ (relative to $\Gamma'^{\mathsf{st}})$ for which $D = 1$. So, the information $(\Gamma^{\mathsf{st}},\Gamma,s,\Gamma' \to X)$ in the definition of a $\theta$-equidimensional flow is equivalent to the ellpitical information $s:\Gamma' \to X$. The motivation for studying quadruples is that this setup can be adapted to situations where $\theta$ is a stability condition on an unstable graph $\Gamma^u$; then $\Gamma$ is a quasi-stable model of $\Gamma^u$, but can no longer be recovered from $\Gamma'$, and $\Gamma^u,\Gamma$ must be put as part of the data.
\end{remark}

\begin{lemma}
The preimage of $\Sigma_{\mathbf{Div}_{g,A}^\theta}$ in $\Sigma_{\mathbf{Ord}_{g,n}}$ is $\Sigma_{\mathbf{Ord}_{g,A}^\theta}
$, and the preimage of $\Sigma_{\mathbf{Ord}_{g,A}^\theta}$ in $\Sigma_{\mathbf{Rub}_{g,n}}$ is $\Sigma_{\mathbf{Rub}_{g,A}^\theta}
$. 
 
\end{lemma}

\begin{proof}
The proof is immediate from the definition of the morphisms. 
\end{proof}

The space $\Sigma_{\mathbf{Div}^\theta_{g,A}}$ has a stabilization map to $\Mgn^\textup{trop}$. It is shown in theorem \cite[Theorem 23]{HMPPS} that this map factors isomorphically through a subdivision. It follows that 
\begin{corollary}
\label{cor:indexsubstructure}
The map 
$$\Sigma_{\mathbf{Rub}_{g,A}^\theta} \to \Mgn^{\textup{trop}}$$
factors isomorphically through the composition of a subdivision and a finite index sublattice inclusion. 
\end{corollary}

\section{Algebraizing and Globalizing}
\subsection{Tropicalization and Tropical Operations}
In this paper, we use the language of logarithmic geometry as our main means to access algebrogeometric problems via combinatorial tools.  We rapidly review the foundations. For a thorough treatment of logarithmic geometry, we refer the interested reader to \cite{K},\cite{Ogus}. For a shallower treatment more adapted to our needs here we refer to \cite{HMPPS,MRan}.   

For any log scheme (or algebraic stack with log structure) $S$, one can assign a cone stack $\Sigma_S$, called the tropicalization of $S$ \cite{CCUW}. The reader can think of as a first approximation that $\Sigma_S$ is a rational polyhedral cone complex here, obtained by gluing the various cones dual to the characteristic monoids 
$$
\sigma_s = \overline{M}_{S,s}^\vee
$$
as $s$ ranges through geometric points of $S$. This may seem like an infinite amount of data, but\footnote{with some mild finiteness assumptions on the underlying scheme of $S$.} the characteristic monoid $\overline{M}_S$ stratifies $S$ into a finite number of strata, and $\overline{M}_{S,s}$ is locally constant on strata, so only a finite number of distinct cones $\sigma_s$ appear in the gluing -- one ``common" cone $\sigma_x$ for each stratum $x$, equal to $\sigma_s$ for any $s \in x$. The only subtlety is that the sheaf $\overline{M}_{S}$ is only defined on the \'etale site of $S$ in general, and so in order to capture this phenomenon the gluing needs to be taken as a stack. What working with the cone stack means in practice, is that there is one cone $\sigma_x$ for each stratum $x$ of $S$ in $\Sigma_S$, as above, but one map $\sigma_y \to \sigma_x$ \emph{for each} \'etale specialization $\zeta: y \to x$ of strata. Again, with mild assumptions on $S$, there is only a finite number of distinct maps $\sigma_y \to \sigma_x$ for such specializations; however, one can have for instance non-trivial specializations $x \rightsquigarrow x$, which amount to automorphisms in $\Sigma_S$. All in all, one gets a system of rational polyhedral cones $\{\sigma_x\}$ with potentially several maps 
\[
\sigma_y \to \sigma_x,
\]
each consisting of a composition of an automorphism with the inclusion of a face. The cone stack $\Sigma_S$ is the colimit 
\[
\Sigma_S = \varinjlim \sigma_x
\]
of this system in the category of stacks over rational polyhedral cones. 

Thus, the cone stack $\Sigma_S$ is in its essence a a tropical object. Nevertheless, any cone stack can be given a functor of points in the category of log schemes: for any cone stack $\Sigma$, one obtains a prestack
$$
\Sigma \rightarrow \textbf{LogSch}
$$
by defining
$$
T \to \Sigma_S \rightsquigarrow \mathsf{Hom}(\overline{M}_T(T)^\vee, \Sigma_S).
$$
The associated stack is representable by an algebraic stack with logarithmic structure, refered to as an \emph{Artin fan}. Instead of introducing Artin fans here, we will simply understand morphisms from a log scheme to $\Sigma$ to mean maps to the associated stack. In particular, we have a tautological morphism 
$$
S \to \Sigma_S
$$
and any map $\Sigma \to \Sigma_S$ induces a map
$$
S \times_{\Sigma_S} \Sigma \to S.
$$
One can thus lift combinatorial operations one performs on $\Sigma_S$ to operations of $S$; often, they have geometric meaning. The ones we've encountered in section \ref{sec: tropical moduli} are: 

\begin{itemize}
    \item Subdivisions $\widetilde{\Sigma}_S \to \Sigma_S$. These lift to \emph{log modifications} $\widetilde{S} \to S$, which are proper and surjective representable maps. 
    \item Roots $\Sigma_S' \to \Sigma_S$, which are maps that replace the integral structure of $\Sigma_S$ with an integral structure coming from a finite index sublattice. These correspond to root stacks of $S$ algebraically, which are proper, non-representable maps $S' \to S$, which are bijective on geometric points -- a generalization of roots along divisors of $S$. Details can be found in \cite{BV},\cite{GMtor}.  
    \item Inclusion of a subcomplex $\Sigma \subset \Sigma_S$. These lift to open inclusions $V \subset S$. 
    
\end{itemize}

\subsection{Logarithmic Curves} Let $C \to S$ be a logarithmic curve. Applying the tropicalization process yields a map of cone stacks 
$$
\Sigma_C \to \Sigma_S.
$$
Unwinding the definition, one arrives at the data of a \emph{family} of tropical curves\footnote{Unwinding the definition is not necessarily simple; we advise the reader with not much experience working with cone complexes to take the following set of data as the \emph{definition} of the tropicalization of $C \to S$.}: 

\begin{itemize}
\item  For each point $s \in S$, an underlying graph $\Gamma_s$: the dual graph of $C_s$. 
\item A tropical curve structure on $\Gamma_s$ metrized by $\overline{M}_{S,s}$: for each edge $e \in E(\Gamma_s)$, a length $\ell_e \in \overline{M}_{S,s}$. The length $\ell_e$ is the ``smoothing parameter" of the corresponding node $q$ in $C_s$: there is a unique element $\ell_e$ in $\overline{M}_{S,s}$ such that 
$$
\overline{M}_{C,q} \cong \overline{M}_{S,s}
\oplus_\NN \NN^2$$
under the map $\NN \to \overline{M}_{S,s}$ sending $1 \to \ell_e$, and the diagonal $\NN \to \NN^2$.
\item Compatibility with \'etale specializations: for each $\zeta: s \rightsquigarrow t$, one has homomorphisms $f_\zeta:\Gamma_t \to \Gamma_s$ compactible with the induced map $\overline{M}_{S,t} \to \overline{M}_{S,s}$. 
\end{itemize}

The geometric notions on tropical curves discussed in the previous section globalize to logarithmic curves. The globalization works the same way for all concepts, by working fiber by fiber and demanding compatibility with \'etale specializations: a geometric object $A$ on a tropical curve globalizes to the analogous object on a logarithmic curve as a system of $A_s$ on $\Gamma_s/\overline{M}_{S,s}$ for each $s \in S$, compatibly with \'etale specializations. For example, a piecewise linear function on $C \to S$ is a collection of piecewise linear functions 
$$
\alpha_s \in \mathsf{PL}(\Gamma_s)
$$
which are compatible with the maps $\Gamma_t \to \Gamma_s$ for each \'etale specialization $s \rightsquigarrow t$. We highlight that in order to define a numerical stability condition $C \to S$, this means that one has to give a stability condition $\theta$ on the dual graph of every fiber, which are compatible with both potential automorphisms of the fibers $C_s$ and which are additive with respect to smoothings of nodes/edge contractions. 

\begin{example}
\label{example: Adivisor}
Let $C \to S$ be a logarithmic curve with $n$ markings $x_1,\cdots,x_n$, and let $(a_1,\cdots,a_n)$ be a vector of integers (for instance, one adding up to $k(2g-2)$). Consider the vector whose $i$-th entry is the sum of $k$-times the degree of the canonical bundle on the component containing the $i$-th marking with $a_i$: 
$$
k\deg \omega_{C/S}(v_i)+a_i
$$

This vector can be considered as a tropical divisor $A$ on $C \to S$: for each $s \in S$, it is the tropical divisor on $\Gamma_s/\overline{M}_{S,s}$ given by 
$$
A= \sum (a_i + k\deg \omega_{C_s}(v_i))v_i
$$
for $v_i$ the vertex containing the marking $x_i$. It is easy to see that this system of tropical divisors specializes correctly under \'etale specializations. 
\end{example}

Subdivisions deserve a special mention. Subdivisions of the fibers $\Gamma_s/\overline{M}_{S,s}$ compatible with \'etale specializations -- give rise to a \emph{logarithmic modifications} $C' \to C$. However, the log modifications that arise this way are special, as the induced map $C' \to S$ remains a logarithmic curve. We call a log modification that arises as a subdivision of tropical curve a subdivision of $C$. 

\section{Algebraic Moduli}
\label{sec: algebraicmoduli}


In this section, we will assume that all our curves come with at least one marking. This is analogous to the assumption in \ref{sec: tropical moduli} and can be avoided, but nevertheless we require it in order to simplify the presentation, as in our applications a marking is always present. We follow \cite{MarcusWise} and define: 

\begin{definition}
The stack $\textbf{Div}$ on $\textbf{LogSch}$ parametrizing over $S$ pairs $(C \to S,\alpha)$, consisting of 
\begin{itemize}
\item A logarithmic curve $C \to S$. 
\item A piecewise linear function $\alpha$ on $C$, which is $0$ on the component containing the first marking.
\end{itemize}
Automorphisms are automorphisms of $\psi:C \to C$ fixing the underlying scheme of $S$, such that the induced automorphism $\overline{\psi}$ on the tropicalization of $C$ respects $\alpha$: $\overline{\psi} \circ \alpha  = \alpha$. 
\end{definition}

\begin{definition}
The stack $\textbf{Ord}$ parametrizing pairs $(C \to S, \alpha)$ of 
\begin{itemize}
\item A log curve $C \to S$. 
\item A piecewise linear function on $C$ whose values are totally ordered and which is $0$ on the component containing the first marking.
\end{itemize} 
\end{definition}

\begin{definition}
The stack $\textbf{Rub}$ parametrizing pairs $(C \to S, \alpha, C' \to X)$ of 
\begin{itemize}
\item A log curve $C \to S$. 
\item A piecewise linear function $\alpha$ on $C$ which is $0$ on the component containing the first marking.
\item A stable lift of $\alpha$ to an equidimensional map 
\[
C' \to X.
\]
\end{itemize} 
\end{definition}
Automorphisms are defined as for $\mathbf{Div}$.

\begin{remark}
The assumption that curves $C \to S$ carry a marking is put precisely in order to rigidify the stacks. We note however that in the case of $\textbf{Rub}$, there is a canonical rigidification which does not depend on the presence of markings: the values of the function $\alpha$ are totally ordered, so we can always demand that $\alpha$ is $0$ on the minimal value.  
\end{remark}

There are evident forgetful maps $\mathbf{Rub} \to \mathbf{Ord} \to \mathbf{Div}$, and a forgetful-stabilization morphisms $\mathbf{Div} \to \Mgn$. We note here that we are potentially in an uncomfortable situation, as in section \ref{sec: tropical moduli} we introduced cone stacks $\Sigma_{\mathbf{Div}},\Sigma_{\mathbf{Ord}},\Sigma_{\mathbf{Rub}}$, whereas we should have reserved the notation for the cone stacks associated to $\mathbf{Div},\mathbf{Rub},\mathbf{Ord}$. However, 
\begin{theorem}
We have 
$$
\mathbf{Div} = \Mgn \times_{\Mgn^{\textup{trop}}} \Sigma_{\mathbf{Div}}
$$
and similar equalities hold for $\mathbf{Ord},\mathbf{Rub}$. 
\end{theorem}

\begin{proof}
We construct a map $\mathbf{Div} \to \Sigma_{\mathbf{Div}}$. Let $S$ be a log scheme; we must construct a map $\mathbf{Div}(S) \to \Sigma_{\mathbf{Div}}(S)$, which amounts to constructing a map locally around each $x \in S$, compatibly with any \'etale specialization $\zeta: x \rightsquigarrow y$. So we may replace $S$ with a sufficiently small neighborhood of $x$. Then, we can assume that $x$ is in the closed stratum of $S$ and that $\overline{M}_{S,x} = \overline{M}_{S}(S)$. Furthermore, $\overline{M}_{S,y}$ is a quotient of $\overline{M}_{S,x}$ by a face. Let $(C \to S,\alpha)$ be an element of $\mathbf{Div}(S)$. 

We write $\Gamma_x$ for the dual graph of $C_x$, and $\Gamma_y$ for the dual graph of $C_y$. These have the structure of tropical curves metrized by $\overline{M}_{S,x}$, $\overline{M}_{S,y}$ respectively. Furthermore, they carry piecewise linear functions $\alpha_x,\alpha_y$. 

The map $\overline{M}_{S,x} \to \overline{M}_{S,y}$ canonically induces a tropical curve $\overline{\Gamma}_x$ metrized by $\overline{M}_{S,y}$, by contracting edges whose length is $0$ in $\overline{M}_{S,y}$, and a piecewise linear function $\overline{\alpha}_x$. We thus get two maps $\overline{M}_{S,y} \to \Sigma_{\mathbf{Div}}$, corresponding to $(\overline{\Gamma}_x,\overline{\alpha}_x)$ and $(\Gamma_y,\alpha_y)$. 

The specialization $\zeta$ induces a map $\overline{\Gamma}_x \to \Gamma_y$, which, by definition of a piecewise linear function on $C$, takes $\overline{\alpha}_x$ to $\alpha_y$. This is precisely an isomorphism in $\Sigma_{\mathbf{Div}}$, and so the map is compatible with specializations. Thus, we get the desired map $\mathbf{Div} \to \Sigma_{\mathbf{Div}}$ and as a result a map 
$$
\mathbf{Div} \to \Mgn \times_{\Mgn^{\textup{trop}}} \Sigma_{\mathbf{Div}}.
$$
On the other hand, an element of the fiber product is a log curve $C \to S$, together with a piecewise linear function on its tropicalization; and so the map $\mathbf{Div} \to \Mgn \times_{\Mgn^{\textup{trop}}} \Sigma_{\mathbf{Div}}$ is essentially surjective. Let $S$ be the spectrum of an algebraically closed field, $C \to S$ a log curve, and $\Gamma$ its tropicalization, metrized by $M = \overline{M}_S(S)$. The automorphism groups
$$
\mathsf{Aut}(\Mgn \times_{\Mgn^{\textup{trop}}} \Sigma_{\mathbf{Div}})(S)
$$
in the fiber product consist of pairs of an automorphism $\phi:\Gamma \to \Gamma$ with $\phi \circ \alpha = \alpha$, together with an automorphism $\psi$ of $C$ inducing $\phi$, i.e. automorphisms $\psi$ of $C$ with $\overline{\psi} \circ \alpha = \alpha$. These are exactly the automorphisms of $\mathbf{Div}$. 

The statement for $\textbf{Ord},\textbf{Rub}$ are proved precisely the same way. 
\end{proof}

As corollaries of the tropical results of \ref{sec: tropical moduli}, we obtain several theorems, by translating the tropical results of section \ref{sec: tropical moduli}. Fixing a universal stability condition $\theta$ and a vector of integers $(a_1,\cdots,a_n)$ with $\sum a_i = k(2g-2)$, and let $A$ be the tropical divisor of example \ref{example: Adivisor}. We obtain:   

\begin{theorem}[Marcus-Wise Corollary 5.3.5]
The maps $\textbf{Rub} \to \mathbf{Ord} \to \mathbf{Div}$ are proper, log \'etale monomorphisms. 
\end{theorem}

\begin{theorem}
The stack $\textbf{Rub}$ is non-singular and its universal curve is quasi-smooth. 
\end{theorem}

\begin{theorem}
The map $\textbf{Rub} \to \Mgn$ is of Deligne-Mumford type. The map $\textbf{Rub} \to \textbf{Ord}$ is a relative coarse moduli space over $\Mgn$.
\end{theorem}

\begin{proof}
This follows from corollary \ref{cor: sublatticestructure} and \cite[Proposition 3.2.6]{GMtor} by observing that the cones in $\Sigma_{\textbf{Ord}},\Sigma_{\textbf{Rub}}$ provide local charts for $\textbf{Ord},\textbf{Rub}$.
\end{proof}

\begin{lemma}
The stacks $\textbf{Div}_{g,A}^\theta$, $\textbf{Ord}_{g,A}^\theta$, $\textbf{Rub}_{g,A}^\theta$ are open substacks of $\mathbf{Div},\mathbf{Ord},\mathbf{Rub}$. 
\end{lemma}

On the other hand, we can restrict the natural map $\mathbf{Div}_{g,n} \to \Mgn$ to $\mathbf{Div}_{g,A}^\theta$, and also the composition $\mathbf{Rub}_{g,A}^\theta \to \mathbf{Div}_{g,A}^\theta$. 



For the sake of completeness, we spell out the functor of points of $\textbf{Div}_{g,A}^\theta$ and $\textbf{Rub}_{g,A}^\theta$. It is simpler to do so in the category $\textbf{LogSch}$. For a log scheme $S$, the $S$ points of $\mathbf{Div}_{g,A}^\theta$ consist of 
\begin{itemize}
    \item A quasi-stable log curve $C \to S$. 
    \item A piecewise linear function, vanishing along the first marking, such that 
    $$
    A-\mathsf{div}(\alpha) = D
    $$
    is $\theta$-stable. 
\end{itemize}

The $S$-points of $\textbf{Rub}_{g,A}^\theta$ are, \emph{in addition to the above}, 

\begin{itemize}
    \item A semistable model $C' \to C$ over $S$, a tropical target $X \to S$, and an equidimensional map 
    $$
    C' \to X
    $$
    This data is required to be stable, i.e. $C',X$ are minimal with this property. 
\end{itemize}

We can retrict the map $\mathbf{Div}_{g,n} \to \Mgn$ to its open substack $\mathbf{Div}_{g,A}^\theta$. The content of \ref{cor:indexsubstructure} and the discussion preceeding it is then that the restriction $\mathbf{Div}_{g,A}^\theta \to \Mgn$ factors isomorphically through a log modification $\DMgn^\theta \to \Mgn$; and the restriction of $\mathbf{Rub}_{g,n} \to \Mgn$ to $\mathbf{Rub}_{g,A}^\theta$ factors through a log modification followed by a root $\RMgn^\theta \to \Mgn$. 

\begin{theorem}
\label{theorem: propertiesofrubtheta}
The stack $\RMgn^\theta$ is non-singular, and the map $\RMgn^\theta \to \Mgn$ is proper and of DM-type. The universal curve $\mathcal{C} \to \RMgn^\theta$ is quasi-smooth, and carries a universal line bundle
$$
\mathcal{L} = \omega^k(\sum a_ix_i) \otimes \mathcal{O}(\alpha)
$$
which is $\theta$-stable. 
\end{theorem}
The line bundle $\mathcal{L}$ in particular gives an Abel-Jacobi section 
$$
\RMgn^\theta \to \textup{Pic}^\theta
$$
and can be used to compute the DR cycle: the ``universal DR formula" of \cite{BHPSS} applies, as in \cite[Theorem A]{HMPPS}.

\begin{remark}
One can use the Abel-Jacobi section  
$$
\DMgn^\theta \to \textup{Pic}^\theta
$$
to pull back the multidegree $0$ universal Jacobian $\textup{Jac}$. The resulting space 
$$
\Mgn^\theta \times_{\textup{Pic}^\theta} \textup{Jac}
$$
is the space $M^{\diamond}$ of \cite{Holmes}. Its pullback to $\RMgn^\theta$ is a desingularization, denoted by $\mathsf{tDR}$ in \cite{MRan}. Pulling back further, by replacing $\textup{Jac}$ with its $0$ section $0$ gives $$
\mathsf{DRL} = \DMgn^\theta \times_{\textup{Pic}^\theta}0.
$$
This is a compact locus which supports the cycle $\mathsf{DR}_{g,A}^k$. When $k=0$, the further pullback 
$$
\RMgn^\theta \times_{\textup{Pic}^\theta} 0
$$
can be identified with the rubber version of Bumsig Kim's space of log stable maps to expansions of $\mathbb{P}^1$. In fact, the tropical targets $X$ are the cone stacks associated to expansions. But notice that in the definition of $\textbf{Rub}$, we demand that curves map to tropical targets and not expansions -- it is precisely the point that there a lot more maps to tropical targets than algebraic ones, and it is in this space that we carve out the smooth spaces $\RMgn^\theta$. The space of maps to algebraic targets is of much smaller dimension. 
\end{remark}

\section{Algorithms} Let $\theta$ be a generic stability condition, and $A$ a tropical divisor. We explain how to construct the cone stack of $\RMgn^\theta$ from that of $\Mgn$ algorithmically. The algorithm first constructs $\DMgn^\theta$, precisely as in \cite{HMPPS}. It then finds the cones in $\RMgn^\theta$ by unpacking the discussion of \ref{sec: tropical moduli}. It suffices to work cone by cone in $\Mgn^\textup{trop}$. We thus fix a stable graph $\Gamma$ of type $g,n$. We will write $\Sigma_{\Gamma}$ for the cone corresponding to $\Gamma$ in $\Mgn$, $\Sigma_{\Gamma}^\theta$ for the cone in $\DMgn^\theta$, and $\widetilde{\Sigma}_{\Gamma}^\theta$ for the cone in $\RMgn^\theta$.   
\begin{center}
\underline{Algorithm}:
\end{center}
\begin{itemize}
\item Step 1: List all acyclic flows $s$ on quasi-stable models of $\Gamma$ (without length assignments) with tropical divisor $\mathsf{div}(s) = A -D$. There is a finite number of possible such flows. \\

\item Step 2: For each such flow, find the $x \in \Sigma_\Gamma$ such that $\left \langle s, \gamma \right \rangle_x = 0$ for any $\gamma \in H_1(\Gamma)$. The collection of such $x$ for a specific flow is a cone of $\Sigma_{\Gamma}^\theta$. In other words, $\Sigma_{\Gamma}^\theta$ is the subdivision of $S$ into the cones where the various acyclic flows lift to actual twists. \\

\item Step 3: Over a cone of $\Sigma_\Gamma^\theta$ corresponding to $(\Gamma',\alpha)$, list all possible orderings $\kappa$ extending $\alpha$. Equivalently, lift the data $\Gamma',\alpha$ to stable equidimensional flows $\Gamma'' \to X$. There is, again, only a finite number of such data. \\

\item Step 4: Find the vectors $x \in \Sigma_\Gamma^\theta$ that realize a given order $\kappa$. This means orienting the edges according to $\kappa$ this total order\footnote{contracted edges do not contribute and can be ignored here.}, choosing a minimal vertex $v$, and an oriented path $P_{v \to w}$ for every $w \in \Gamma'$; for all given inequalities $\alpha(w) < \alpha(w)$ in the given order, find $x$ such that 
$$
\sum_{\vec{e} \in P_{v \to w}} s(\vec{e})\ell_e \le \sum_{\vec{e} \in P_{v \to u}} s(\vec{e})\ell_e.
$$
This determines the cones $\sigma \in \widetilde{\Sigma}_{\Gamma}^\theta$.\\

\item Step 5: For each cone $\sigma$ in $\Sigma_{\Gamma}^\theta$, take a generating set for $\sigma \cap N_{\Sigma_{\Gamma}^\theta}$. For $x$ in this generating set, find the minimal integral multiple $kx$ of $x$ for which the quantities 
$$
\frac{\alpha(u)-\alpha(v)}{s(f)},\frac{\alpha(w)-\alpha(v)}{s(f)}
$$
are integers when evaluated on $kx$ where: $v$ ranges through all vertices of $\Gamma'$: $w$ is any vertex with $\alpha(w)$ the value consecutive to $\alpha(v)$; and $f$ is any edge oriented away from $v$ with other endpoint $u$. The sublattice generated by the $kx$ is the integral structure of $\widetilde{\Sigma}_{\Gamma}^\theta$. 

Alternatively, the integral structure can be determined as the sublattice 
$$
\oplus_{\rho \in \sigma(1)}\NN x_{\rho} \subset \NN^{E(\Gamma)}
$$
generated by the primitive vectors $x_\rho$ along the one dimensional faces $\rho$ of $\sigma$. 
\end{itemize}

\begin{remark}
Steps $2$ and $4$ are the hardest steps to carry out, as they involve solving a collection of linear inequalities. However, as is explained in \cite{HMPPS}, the difficulty of Step $2$ is deceptive. In fact, \emph{it suffices} to solve the equations $\left \langle s, \gamma \right \rangle_x = 0$ only for acyclic flows $s$ supported on $\Gamma$ rather than a quasi-stable model. The system of inequalities then simplifies significantly, as it reduces to a linear system of \emph{equalities}. These flows determine the minimal cones of $\Sigma_Gamma^\theta$, and all other cones are determined by how twists specialize -- i.e., via the combinatorics of specializations rather than the tropical geometry. The same is true for step $4$: it suffices to solve the inequalities for total orders with the fewest possible strict inequalities; the other cones are determined by specializations. Thus, tropical geometry only enters to determine the shallowest strata. Afterwards, combinatorics takes over.  
\end{remark}

\section{Example}
We present an example of the construction. We use the ramification vector $A = (-4,3,1)$ on $\overline{\mathcal{M}}_{1,3}$, and work out the subdivision of the cone $\mathbb{R}_{\ge 0}^3 = \left\langle \ell_1,\ell_2,\ell_3 \right\rangle$ corresponding to the triangular graph $\Gamma$ that consists of three vertices $v_1,v_2,v_3$ with three edges $e_1,e_2,e_3$ between them, depicted below: 

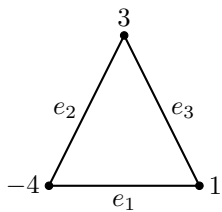
\[
\begin{tikzpicture}
\coordinate(u1) at (-1,0){};
\node[left] at (u1){$-4$};
\node[draw,circle,inner sep=1pt,fill] at (u1){};
\coordinate(u2) at (1,0){};
\node[right] at (u2){$1$};
\node[draw,circle,inner sep=1pt,fill] at (u2){};
\coordinate(u3) at (0,2){};
\node[draw,circle,inner sep=1pt,fill] at (u3){};
\node[above] at (u3){$3$};
\draw[-,thick](u1)--(u2)--(u3)--(u1);
\node[below] at (0,0){$e_1$};
\node[left] at (-1/2,1){$e_2$};
\node[right] at (1/2,1){$e_3$};
\end{tikzpicture}
\]
We choose a small perturbation of $\theta = 0$ which is negative on the component containing the first marking and positive on the others. The acyclic flows balancing $A-D$, as $D$ ranges through $\theta$-stable divisors on $\Gamma$ are then given in the following list (with slopes depicted in red):

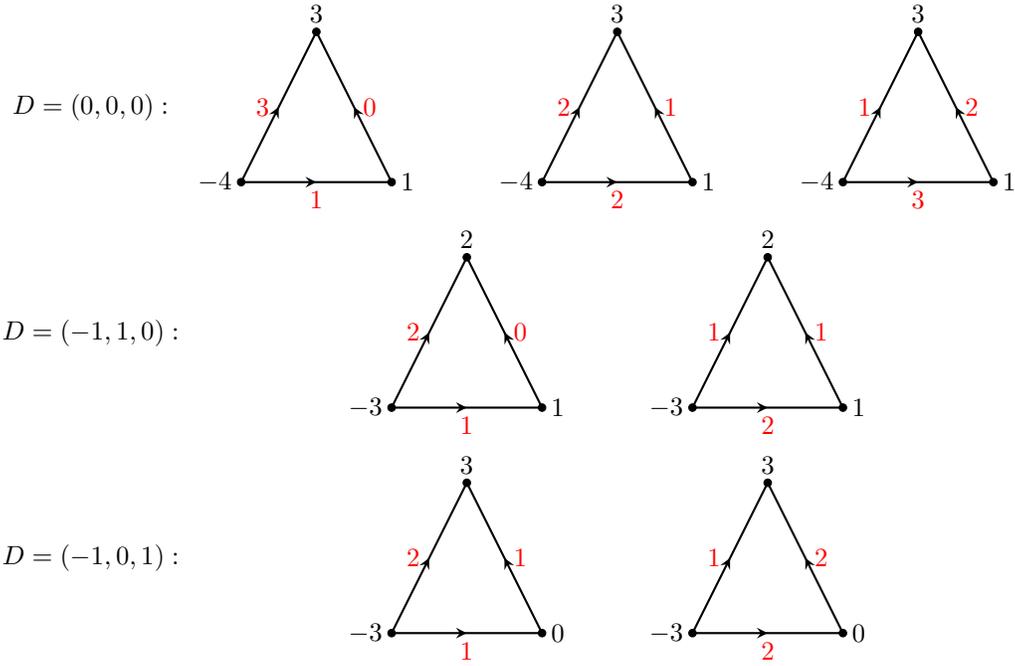
\[
\begin{tikzpicture}
[decoration={markings, 
    mark= at position 0.5 with {\arrow{stealth}},
    }
] 
\node at (-7,1){$D=(0,0,0):$};

\coordinate(u1) at (-1,0){};
\node[left] at (u1){$-4$};
\node[draw,circle,inner sep=1pt,fill] at (u1){};
\coordinate(u2) at (1,0){};
\node[right] at (u2){$1$};
\node[draw,circle,inner sep=1pt,fill] at (u2){};
\coordinate(u3) at (0,2){};
\node[draw,circle,inner sep=1pt,fill] at (u3){};
\node[above] at (u3){$3$};
\draw[-,thick, postaction={decorate}](u1)--(u2);
\draw[-, thick, postaction={decorate}](u2) --(u3); 
\draw[-,thick,postaction={decorate}](u1)--(u3);
\node[below] at (0,0){\textcolor{red}{$2$}};
\node[right] at (.5,1){\textcolor{red}{$1$}}; 
\node[left] at (-.5,1){\textcolor{red}{$2$}};

\coordinate(v1) at (-5,0){};
\node[left] at (v1){$-4$};
\node[draw,circle,inner sep=1pt,fill] at (v1){};
\coordinate(v2) at (-3,0){};
\node[right] at (v2){$1$};
\node[draw,circle,inner sep=1pt,fill] at (v2){};
\coordinate(v3) at (-4,2){};
\node[draw,circle,inner sep=1pt,fill] at (v3){};
\node[above] at (v3){$3$};
\draw[-,thick,postaction={decorate}](v1)--(v2);
\draw[-, thick,postaction={decorate}](v2) --(v3); 
\draw[-,thick,postaction={decorate}](v1)--(v3);

\node[below] at (-4,0){\textcolor{red}{$1$}};
\node[right] at (-3.5,1){\textcolor{red}{$0$}}; 
\node[left] at (-4.5,1){\textcolor{red}{$3$}};


\coordinate(w1) at (3,0){};
\node[left] at (w1){$-4$};
\node[draw,circle,inner sep=1pt,fill] at (w1){};
\coordinate(w2) at (5,0){};
\node[right] at (w2){$1$};
\node[draw,circle,inner sep=1pt,fill] at (w2){};
\coordinate(w3) at (4,2){};
\node[draw,circle,inner sep=1pt,fill] at (w3){};
\node[above] at (w3){$3$};
\draw[-,thick,postaction={decorate}](w1)--(w2);
\draw[-,thick,postaction={decorate}](w2)--(w3);
\draw[-,thick,postaction={decorate}](w1)--(w3);
\node[below] at (4,0){\textcolor{red}{$3$}};
\node[right] at (4.5,1){\textcolor{red}{$2$}}; 
\node[left] at (3.5,1){\textcolor{red}{$1$}};


\node at (-7,-2){$D=(-1,1,0):$};

\coordinate(u1) at (1,-3){};
\node[left] at (u1){$-3$};
\node[draw,circle,inner sep=1pt,fill] at (u1){};
\coordinate(u2) at (3,-3){};
\node[right] at (u2){$1$};
\node[draw,circle,inner sep=1pt,fill] at (u2){};
\coordinate(u3) at (2,-1){};
\node[draw,circle,inner sep=1pt,fill] at (u3){};
\node[above] at (u3){$2$};
\draw[-,thick, postaction={decorate}](u1)--(u2);
\draw[-, thick, postaction={decorate}](u2) --(u3); 
\draw[-,thick,postaction={decorate}](u1)--(u3);
\node[below] at (2,-3){\textcolor{red}{$2$}};
\node[right] at (2.5,-2){\textcolor{red}{$1$}}; 
\node[left] at (1.5,-2){\textcolor{red}{$1$}};

\coordinate(v1) at (-3,-3){};
\node[left] at (v1){$-3$};
\node[draw,circle,inner sep=1pt,fill] at (v1){};
\coordinate(v2) at (-1,-3){};
\node[right] at (v2){$1$};
\node[draw,circle,inner sep=1pt,fill] at (v2){};
\coordinate(v3) at (-2,-1){};
\node[draw,circle,inner sep=1pt,fill] at (v3){};
\node[above] at (v3){$2$};
\draw[-,thick,postaction={decorate}](v1)--(v2);
\draw[-, thick,postaction={decorate}](v2) --(v3); 
\draw[-,thick,postaction={decorate}](v1)--(v3);

\node[below] at (-2,-3){\textcolor{red}{$1$}};
\node[right] at (-1.5,-2){\textcolor{red}{$0$}}; 
\node[left] at (-2.5,-2){\textcolor{red}{$2$}};

\node at (-7,-5){$D=(-1,0,1):$};

\coordinate(u1) at (1,-6){};
\node[left] at (u1){$-3$};
\node[draw,circle,inner sep=1pt,fill] at (u1){};
\coordinate(u2) at (3,-6){};
\node[right] at (u2){$0$};
\node[draw,circle,inner sep=1pt,fill] at (u2){};
\coordinate(u3) at (2,-4){};
\node[draw,circle,inner sep=1pt,fill] at (u3){};
\node[above] at (u3){$3$};
\draw[-,thick, postaction={decorate}](u1)--(u2);
\draw[-, thick, postaction={decorate}](u2) --(u3); 
\draw[-,thick,postaction={decorate}](u1)--(u3);
\node[below] at (2,-6){\textcolor{red}{$2$}};
\node[right] at (2.5,-5){\textcolor{red}{$2$}}; 
\node[left] at (1.5,-5){\textcolor{red}{$1$}};

\coordinate(v1) at (-3,-6){};
\node[left] at (v1){$-3$};
\node[draw,circle,inner sep=1pt,fill] at (v1){};
\coordinate(v2) at (-1,-6){};
\node[right] at (v2){$0$};
\node[draw,circle,inner sep=1pt,fill] at (v2){};
\coordinate(v3) at (-2,-4){};
\node[draw,circle,inner sep=1pt,fill] at (v3){};
\node[above] at (v3){$3$};
\draw[-,thick,postaction={decorate}](v1)--(v2);
\draw[-, thick,postaction={decorate}](v2) --(v3); 
\draw[-,thick,postaction={decorate}](v1)--(v3);

\node[below] at (-2,-6){\textcolor{red}{$1$}};
\node[right] at (-1.5,-5){\textcolor{red}{$1$}}; 
\node[left] at (-2.5,-5){\textcolor{red}{$2$}};

\end{tikzpicture}
\]

The seven flows above contribute seven codimension $g=1$ cells in the subdivision $S^\theta$ of $\mathbb{R}_{\ge 0}^3$. For example, the first listed flow contributes the wall 
$$
\ell_1 = 3\ell_2.
$$
In general, to complete the subdivision, it is necessary to consider all $\theta$-semistable divisors $D$ on quasi-stable models of $\Gamma$ as well. There are $8$ such flows, but in this case we do not need to list them, as since $g=1$ the $7$ codimension $1$ walls in fact determine the subdivision\footnote{This property is unique to $g=1$. See \cite[Section 4.3]{HMPPS} for a higher genus example.}. The desired subdivision looks as follows (we present the induced subdivision of the triangle obtained by cutting $\mathbb{R}_{\ge 0}^3$ with the hyperplane $\ell_1+\ell_2+\ell_3 = 1$): 

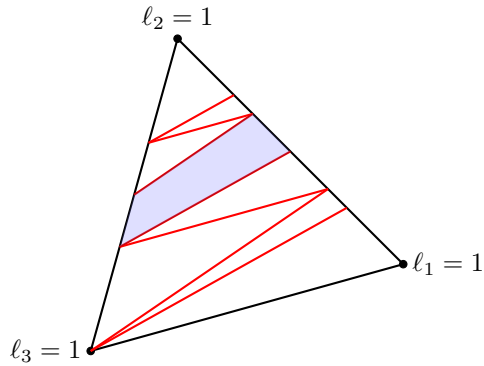
\[
\begin{tikzpicture}[scale=3]
\coordinate(u1) at (1,0,0){};
\node[draw,circle,inner sep=1pt,fill] at (u1){};
\node[right] at (u1){$\ell_1=1$};
\coordinate(u2) at (0,1,0){};
\node[draw,circle,inner sep=1pt,fill] at (u2){};
\node[above] at (u2){$\ell_2=1$};
\coordinate(u3) at (0,0,1){};
\node[draw,circle,inner sep=1pt,fill] at (u3){};
\node[left] at (u3){$\ell_3=1$};
\draw[-,thick](u1)--(u2)--(u3)--(u1);

\coordinate(v1) at (3/4,1/4,0){};
\draw[-,red,thick](v1)--(u3);
\coordinate(v2) at (2/3,1/3,0){};
\draw[-,red,thick](v2)--(u3);
\coordinate(v3) at (0,1/3,2/3){};
\draw[-,red,thick](v2)--(v3);
\coordinate(v4) at (1/2,1/2,0){};
\draw[-,red,thick](v3)--(v4);
\coordinate(v5) at (0,1/2,1/2){};
\coordinate(v6) at (1/3,2/3,0){};
\draw[-,red,thick](v5)--(v6);
\coordinate(v7) at (0,2/3,1/3);
\coordinate(v8) at (1/4,3/4,0);
\draw[-,red,thick](v6)--(v7);
\draw[-,red,thick](v7)--(v8);
\draw[fill=blue!50, nearly transparent](v3)--(v4)--(v6)--(v5)--cycle;

\end{tikzpicture}
\]

The shaded (slice of the) cone is special in $\Sigma_\Gamma^\theta$, as it is not simplicial. It is the region corresponding to the twist 

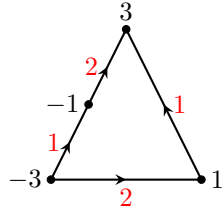
\[
\begin{tikzpicture}
[decoration={markings, 
    mark= at position 0.5 with {\arrow{stealth}},
    }
] 
\coordinate(v1) at (-1,0){};
\node[left] at (v1){$-3$};
\node[draw,circle,inner sep=1pt,fill] at (v1){};
\coordinate(v2) at (1,0){};
\node[right] at (v2){$1$};
\node[draw,circle,inner sep=1pt,fill] at (v2){};
\coordinate(v3) at (0,2){};
\node[draw,circle,inner sep=1pt,fill] at (v3){};
\node[above] at (v3){$3$};
\coordinate(v4) at (-1/2,1){};
\node[left] at (v4){$-1$};
\node[draw,circle,inner sep=1pt,fill] at (v4){};
\draw[-,thick,postaction={decorate}](v1)--(v4);
\draw[-, thick,postaction={decorate}](v4) --(v3); 
\draw[-,thick,postaction={decorate}](v1)--(v2);
\draw[-,thick,postaction={decorate}](v2)--(v3);

\node[left] at (-3/4,1/2){\textcolor{red}{$1$}};
\node[left] at (-1/4,3/2){\textcolor{red}{$2$}}; 
\node[below] at (0,0){\textcolor{red}{$2$}};
\node[right] at (1/2,1){\textcolor{red}{$1$}};

\end{tikzpicture}
\]
The cone is subdivided further in $\widetilde{\Sigma
}_\Gamma^\theta$, according to whether the piecewise linear function $\alpha$ is greater on the exceptional vertex $u$, or the vertex $v_2$. This corresponds to the three ways to make $\alpha$ equidimensional, namely \footnote{we have replaced the degree of $\mathsf{div}(\alpha)$ with the name of the vertex for display purposes}

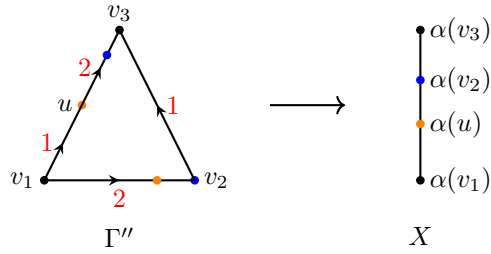
\[
\begin{tikzpicture}
[decoration={markings, 
    mark= at position 0.5 with {\arrow{stealth}},
    }
] 
\coordinate(v1) at (-1,0){};
\node[left] at (v1){$v_1$};
\node[draw,circle,inner sep=1pt,fill] at (v1){};
\coordinate(v2) at (1,0){};
\node[right] at (v2){$v_2$};
\node[draw,blue,circle,inner sep=1pt,fill] at (v2){};
\coordinate(v3) at (0,2){};
\node[draw,circle,inner sep=1pt,fill] at (v3){};
\node[above] at (v3){$v_3$};
\coordinate(v4) at (-1/2,1){};
\node[left] at (v4){$u$};
\node[draw,orange,circle,inner sep=1pt,fill] at (v4){};
\draw[-,thick,postaction={decorate}](v1)--(v4);
\draw[-, thick,postaction={decorate}](v4) --(v3); 
\draw[-,thick,postaction={decorate}](v1)--(v2);
\draw[-,thick,postaction={decorate}](v2)--(v3);

\draw[->, thick](2,1)--(3,1);

\node[left] at (-3/4,1/2){\textcolor{red}{$1$}};
\node[left] at (-1/4,3/2){\textcolor{red}{$2$}}; 
\node[below] at (0,0){\textcolor{red}{$2$}};
\node[right] at (1/2,1){\textcolor{red}{$1$}};
\node[below] at (0,-1/2){$\Gamma''$};

\node[draw,orange,circle,inner sep=1pt, fill] at (1/2,0){};
\node[draw,blue,circle,inner sep=1pt,fill] at (-1/6,10/6){};

\coordinate(x4) at (4,2){};
\coordinate(x3) at (4,4/3){};
\coordinate(x2) at (4,3/4){};
\coordinate(x1) at (4,0){};
\draw[-,thick](x1)--(x4);
\node[right] at (x1){$\alpha(v_1)$};
\node[draw,circle,inner sep=1pt,fill] at (x1){};
\node[below] at (4,-1/2){$X$};

\node[right] at (x4){$\alpha(v_3)$};
\node[draw,circle,inner sep=1pt,fill] at (x4){};

\node[right] at (x3){$\alpha(v_2)$};
\node[draw,blue,circle,inner sep=1pt,fill] at (x3){};

\node[right] at (x2){$\alpha(u)$};
\node[draw,orange, circle,inner sep=1pt,fill] at (x2){};

\end{tikzpicture}
\]

or

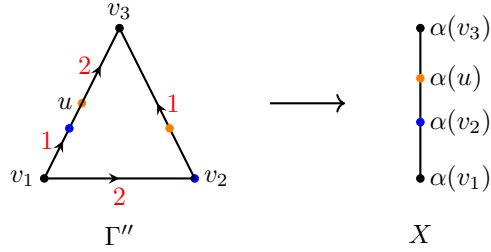
\[
\begin{tikzpicture}
[decoration={markings, 
    mark= at position 0.5 with {\arrow{stealth}},
    }
] 
\coordinate(v1) at (-1,0){};
\node[left] at (v1){$v_1$};
\node[draw,circle,inner sep=1pt,fill] at (v1){};
\coordinate(v2) at (1,0){};
\node[right] at (v2){$v_2$};
\node[draw,blue,circle,inner sep=1pt,fill] at (v2){};
\coordinate(v3) at (0,2){};
\node[draw,circle,inner sep=1pt,fill] at (v3){};
\node[above] at (v3){$v_3$};
\coordinate(v4) at (-1/2,1){};
\node[left] at (v4){$u$};
\node[draw,orange,circle,inner sep=1pt,fill] at (v4){};
\draw[-,thick,postaction={decorate}](v1)--(v4);
\draw[-, thick,postaction={decorate}](v4) --(v3); 
\draw[-,thick,postaction={decorate}](v1)--(v2);
\draw[-,thick,postaction={decorate}](v2)--(v3);

\draw[->, thick](2,1)--(3,1);

\node[left] at (-3/4,1/2){\textcolor{red}{$1$}};
\node[left] at (-1/4,3/2){\textcolor{red}{$2$}}; 
\node[below] at (0,0){\textcolor{red}{$2$}};
\node[right] at (1/2,1){\textcolor{red}{$1$}};
\node[below] at (0,-1/2){$\Gamma''$};

\node[draw,orange,circle,inner sep=1pt, fill] at (2/3,2/3){};
\node[draw,blue,circle,inner sep=1pt,fill] at (-2/3,2/3){};

\coordinate(x4) at (4,2){};
\coordinate(x3) at (4,4/3){};
\coordinate(x2) at (4,3/4){};
\coordinate(x1) at (4,0){};
\draw[-,thick](x1)--(x4);
\node[right] at (x1){$\alpha(v_1)$};
\node[draw,circle,inner sep=1pt,fill] at (x1){};
\node[below] at (4,-1/2){$X$};

\node[right] at (x4){$\alpha(v_3)$};
\node[draw,circle,inner sep=1pt,fill] at (x4){};

\node[right] at (x2){$\alpha(v_2)$};
\node[draw,orange,circle,inner sep=1pt,fill] at (x3){};

\node[right] at (x3){$\alpha(u)$};
\node[draw,blue, circle,inner sep=1pt,fill] at (x2){};

\end{tikzpicture}
\]

or, the degenerate case (note that no subdivision of $\Gamma'$ is required here)

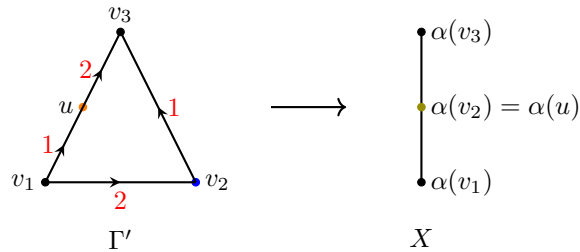
\[
\begin{tikzpicture}
[decoration={markings, 
    mark= at position 0.5 with {\arrow{stealth}},
    }
] 
\coordinate(v1) at (-1,0){};
\node[left] at (v1){$v_1$};
\node[draw,circle,inner sep=1pt,fill] at (v1){};
\coordinate(v2) at (1,0){};
\node[right] at (v2){$v_2$};
\node[draw,blue,circle,inner sep=1pt,fill] at (v2){};
\coordinate(v3) at (0,2){};
\node[draw,circle,inner sep=1pt,fill] at (v3){};
\node[above] at (v3){$v_3$};
\coordinate(v4) at (-1/2,1){};
\node[left] at (v4){$u$};
\node[draw,orange,circle,inner sep=1pt,fill] at (v4){};
\draw[-,thick,postaction={decorate}](v1)--(v4);
\draw[-, thick,postaction={decorate}](v4) --(v3); 
\draw[-,thick,postaction={decorate}](v1)--(v2);
\draw[-,thick,postaction={decorate}](v2)--(v3);

\draw[->, thick](2,1)--(3,1);

\node[left] at (-3/4,1/2){\textcolor{red}{$1$}};
\node[left] at (-1/4,3/2){\textcolor{red}{$2$}}; 
\node[below] at (0,0){\textcolor{red}{$2$}};
\node[right] at (1/2,1){\textcolor{red}{$1$}};
\node[below] at (0,-1/2){$\Gamma'$};

\coordinate(x4) at (4,2){};
\coordinate(x3) at (4,1){};
\coordinate(x2) at (4,1){};
\coordinate(x1) at (4,0){};
\draw[-,thick](x1)--(x4);
\node[right] at (x1){$\alpha(v_1)$};
\node[draw,circle,inner sep=1pt,fill] at (x1){};
\node[below] at (4,-1/2){$X$};

\node[right] at (x4){$\alpha(v_3)$};
\node[draw,circle,inner sep=1pt,fill] at (x4){};

\node[right] at (x2){$\alpha(v_2)=\alpha(u)$};
\node[draw,olive,circle,inner sep=1pt,fill] at (x3){};

\end{tikzpicture}
\]

More concretely, if the edge $e_2$ has been subdivided as $\ell_2' + \ell_2''$, with $\ell_2'$ the length of the edge connecting $v_1$ to the exceptional vertex, the cone is subdivided along the hyperplane 
$$
\ell_2' = 2\ell_1.
$$
This yields the simplicial subdivision
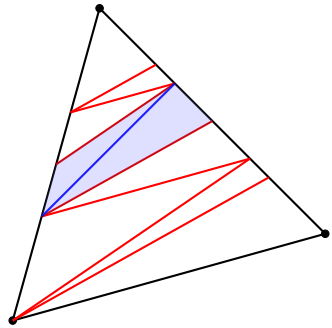
\[
\begin{tikzpicture}[scale=3]
\coordinate(u1) at (1,0,0){};
\node[draw,circle,inner sep=1pt,fill] at (u1){};
\coordinate(u2) at (0,1,0){};
\node[draw,circle,inner sep=1pt,fill] at (u2){};
\coordinate(u3) at (0,0,1){};
\node[draw,circle,inner sep=1pt,fill] at (u3){};
\draw[-,thick](u1)--(u2)--(u3)--(u1);

\coordinate(v1) at (3/4,1/4,0){};
\draw[-,red,thick](v1)--(u3);
\coordinate(v2) at (2/3,1/3,0){};
\draw[-,red,thick](v2)--(u3);
\coordinate(v3) at (0,1/3,2/3){};
\draw[-,red,thick](v2)--(v3);
\coordinate(v4) at (1/2,1/2,0){};
\draw[-,red,thick](v3)--(v4);
\coordinate(v5) at (0,1/2,1/2){};
\coordinate(v6) at (1/3,2/3,0){};
\draw[-,red,thick](v5)--(v6);
\coordinate(v7) at (0,2/3,1/3);
\coordinate(v8) at (1/4,3/4,0);
\draw[-,red,thick](v6)--(v7);
\draw[-,red,thick](v7)--(v8);
\draw[-,blue,thick](v3)--(v6);
\draw[fill=blue!50, nearly transparent](v3)--(v4)--(v6)--(v5)--cycle;

\end{tikzpicture}
\]

Next, we determine the integral structure of the two maximal cones. In fact, the one corresponding to the second equidimensional twist has the induced integral structure, so we describe it only for the first twist

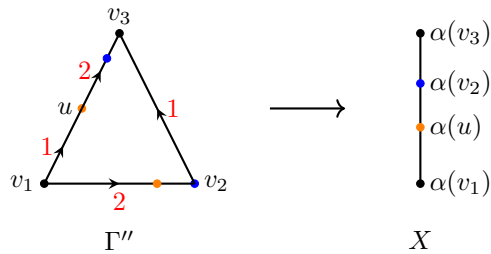
\[
\begin{tikzpicture}
[decoration={markings, 
    mark= at position 0.5 with {\arrow{stealth}},
    }
] 
\coordinate(v1) at (-1,0){};
\node[left] at (v1){$v_1$};
\node[draw,circle,inner sep=1pt,fill] at (v1){};
\coordinate(v2) at (1,0){};
\node[right] at (v2){$v_2$};
\node[draw,blue,circle,inner sep=1pt,fill] at (v2){};
\coordinate(v3) at (0,2){};
\node[draw,circle,inner sep=1pt,fill] at (v3){};
\node[above] at (v3){$v_3$};
\coordinate(v4) at (-1/2,1){};
\node[left] at (v4){$u$};
\node[draw,orange,circle,inner sep=1pt,fill] at (v4){};
\draw[-,thick,postaction={decorate}](v1)--(v4);
\draw[-, thick,postaction={decorate}](v4) --(v3); 
\draw[-,thick,postaction={decorate}](v1)--(v2);
\draw[-,thick,postaction={decorate}](v2)--(v3);

\draw[->, thick](2,1)--(3,1);

\node[left] at (-3/4,1/2){\textcolor{red}{$1$}};
\node[left] at (-1/4,3/2){\textcolor{red}{$2$}}; 
\node[below] at (0,0){\textcolor{red}{$2$}};
\node[right] at (1/2,1){\textcolor{red}{$1$}};
\node[below] at (0,-1/2){$\Gamma''$};

\node[draw,orange,circle,inner sep=1pt, fill] at (1/2,0){};
\node[draw,blue,circle,inner sep=1pt,fill] at (-1/6,10/6){};

\coordinate(x4) at (4,2){};
\coordinate(x3) at (4,4/3){};
\coordinate(x2) at (4,3/4){};
\coordinate(x1) at (4,0){};
\draw[-,thick](x1)--(x4);
\node[right] at (x1){$\alpha(v_1)$};
\node[draw,circle,inner sep=1pt,fill] at (x1){};
\node[below] at (4,-1/2){$X$};

\node[right] at (x4){$\alpha(v_3)$};
\node[draw,circle,inner sep=1pt,fill] at (x4){};

\node[right] at (x3){$\alpha(v_2)$};
\node[draw,blue,circle,inner sep=1pt,fill] at (x3){};

\node[right] at (x2){$\alpha(u)$};
\node[draw,orange, circle,inner sep=1pt,fill] at (x2){};

\end{tikzpicture}
\]

The three rays of the corresponding cone are obtained by specializing the values $\alpha(v_1),\alpha(u),\alpha(v_2),\alpha(v_3)$ so that only two of them are distinct. There are thus three ways to do so, respecting the given total order $\alpha(v_1) \le \alpha(u) \le \alpha(v_2) \le \alpha(v_3)$:

\begin{align*}
\alpha(v_1) = \alpha(u) = \alpha(v_2) < \alpha(v_3)\\
\alpha(v_1) < \alpha(u) = \alpha(v_2) = \alpha(v_3)\\
\alpha(v_1) = \alpha(u) < \alpha(v_2) = \alpha(v_3)
\end{align*}

The first of these is the specialization 

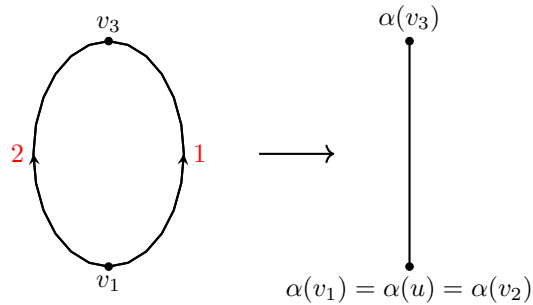
\[
\begin{tikzpicture}
[decoration={markings, 
    mark= at position 0.5 with {\arrow{stealth}},
    }
]
    \draw [thick,domain=-180:180,postaction={decorate}] plot ({cos(\x)},{3/2*sin(\x)});

    \draw [thick,domain=180:-180,postaction={decorate}] plot ({-cos(\x)},{-3/2*sin(\x)});

    \draw[->,thick](2,0)--(3,0);

    \draw[thick,domain=-3/2:3/2]
    plot ({4},{\x}); 

\coordinate(x1) at (0,3/2){}; 
\coordinate(x2) at (0,-3/2){};
\coordinate(y1) at (4,3/2){};
\coordinate(y2) at (4,-3/2){};
\coordinate(z1) at (1,0){};
\coordinate(z2) at (-1,0){};

\node[draw,circle,inner sep=1pt,fill] at (x1){};
\node[above] at (x1){$v_3$};

\node[draw,circle,inner sep=1pt,fill] at (x2){};
\node[below] at (x2){$v_1$};

\node[draw,circle,inner sep=1pt,fill] at (y1){};
\node[above] at (y1){$\alpha(v_3)$};

\node[draw,circle,inner sep=1pt,fill] at (y2){};
\node[below] at (y2){$\alpha(v_1)=\alpha(u)=\alpha(v_2)$};

\node[right] at (z1){\textcolor{red}{$1$}};

\node[left] at (z2){\textcolor{red}{$2$}};

\end{tikzpicture}
\]
Here, the subdivision demanded that 
$$
\ell_2 = \ell_2' + \ell_2''
$$
and 
$$
\ell_2' \le 2 \ell_1.
$$
Since $\ell_1$ is contracted, this leads to the relation 
$$
\ell_2 = \ell_2''
$$
where the slope of $\alpha$ is $2$, and so to the equation 
$$
2\ell_2 = \ell_3.
$$
Thus, the specialization is the ray through the point  
$$
(0,1,2).
$$
The other specializations are obtained similiarily, leading to the points $(1,2,0),(1,1,0)$. The integral structure of the cone is then 
$$
\NN(0,1,2) \oplus \NN(1,2,0) \oplus \NN(1,1,0) \subset \NN^3.
$$
Its index is computed as the determinant of the three vectors, and is found to be $2$.

\bibliography{refs2}
\vspace{8pt}

\noindent Departement Mathematik, ETH Z\"urich\\
\noindent Rämistrasse 101, 8092 Zürich, Switzerland\\
\noindent samouil.molcho@math.ethz.ch

\end{document}